\documentclass[11pt]{amsart}

\usepackage{subfiles}
\usepackage{euscript}
\usepackage{tabu}
\usepackage[margin=1in]{geometry}
\usepackage[dvipsnames]{xcolor} 		             		
\usepackage{graphicx}			
\usepackage{amssymb}
\usepackage{mathrsfs}
\usepackage{amsthm}
\usepackage{amsmath}
\usepackage{stmaryrd}

\usepackage{tikz}
\usepackage{tikz-cd}
\usetikzlibrary{arrows,arrows.meta}
\usetikzlibrary{arrows.meta,        
                decorations.pathmorphing,
                matrix}%
\tikzset{
  commutative diagrams/.cd, 
  arrow style=tikz, 
  diagrams={>=cm to}
}

\usepackage{accents}
\usepackage{upgreek}
\usepackage{enumerate}
\usepackage{bm}
\usepackage{mathtools}
\usepackage[all]{xy}
\usepackage{caption}

\usepackage{url}
\usepackage{float}
\usepackage{todonotes}
\usepackage{bbm}
\usepackage{colonequals}
\usepackage{longtable}

\usepackage[full]{textcomp}
\usepackage[cal=cm]{mathalfa}
\usepackage{xparse}
\usepackage{comment}
\usepackage{cite}

\usepackage[normalem]{ulem}

\theoremstyle{definition}
\newtheorem{innercustomthm}{Theorem}
\newenvironment{customthm}[1]
  {\renewcommand\theinnercustomthm{#1}\innercustomthm}
  {\endinnercustomthm}
  
\theoremstyle{definition}
\newtheorem{innercustomcor}{Corollary}
\newenvironment{customcor}[1]
  {\renewcommand\theinnercustomcor{#1}\innercustomcor}
  {\endinnercustomcor}
  
\theoremstyle{definition}
\newtheorem{innercustomconj}{Conjecture}
\newenvironment{customconj}[1]
  {\renewcommand\theinnercustomconj{#1}\innercustomconj}
  {\endinnercustomconj}
  
\makeatletter
\def\@tocline#1#2#3#4#5#6#7{\relax
  \ifnum #1>\c@tocdepth 
  \else
    \par \addpenalty\@secpenalty\addvspace{#2}%
    \begingroup \hyphenpenalty\@M
    \@ifempty{#4}{%
      \@tempdima\csname r@tocindent\number#1\endcsname\relax
    }{%
      \@tempdima#4\relax
    }%
    \parindent\z@ \leftskip#3\relax \advance\leftskip\@tempdima\relax
    \rightskip\@pnumwidth plus4em \parfillskip-\@pnumwidth
    #5\leavevmode\hskip-\@tempdima
      \ifcase #1
       \or\or \hskip 1em \or \hskip 2em \else \hskip 3em \fi%
      #6\nobreak\relax
    \dotfill\hbox to\@pnumwidth{\@tocpagenum{#7}}\par
    \nobreak
    \endgroup
  \fi}
\makeatother

\makeatletter
\DeclareRobustCommand{\cev}[1]{%
  \mathpalette\do@cev{#1}%
}
\newcommand{\do@cev}[2]{%
  \fix@cev{#1}{+}%
  \reflectbox{$\m@th#1\vec{\reflectbox{$\fix@cev{#1}{-}\m@th#1#2\fix@cev{#1}{+}$}}$}%
  \fix@cev{#1}{-}%
}
\newcommand{\fix@cev}[2]{%
  \ifx#1\displaystyle
    \mkern#23mu
  \else
    \ifx#1\textstyle
      \mkern#23mu
    \else
      \ifx#1\scriptstyle
        \mkern#22mu
      \else
        \mkern#22mu
      \fi
    \fi
  \fi
}

\makeatother

\usetikzlibrary{calc}
\usetikzlibrary{fadings}
\usetikzlibrary{decorations.pathmorphing}
\usetikzlibrary{decorations.pathreplacing}

\newcounter{marginnote}
\DeclareMathAlphabet{\mathpzc}{OT1}{pzc}{m}{it}

\usepackage[backref=page]{hyperref}
\hypersetup{
  colorlinks   = true,          
  urlcolor     = blue,          
  linkcolor    = violet,          
  citecolor   = blue             
}

\newcommand{\lu}[1]{{\color{black} #1}}

\theoremstyle{definition}
\newtheorem{theorem}{Theorem}[section]

\newtheorem{corollary}[theorem]{Corollary}
\newtheorem{lemma}[theorem]{Lemma}
\newtheorem{proposition}[theorem]{Proposition}
\newtheorem{remark}[theorem]{Remark}

\newtheorem*{runningexample*}{Running example}

\newtheorem*{aside*}{Aside}

\newtheorem{construction}[theorem]{Construction}

\newtheorem{definition}[theorem]{Definition}
\newtheorem{example}[theorem]{Example}

\newtheorem{proposition-definition}[theorem]{Proposition-Definition}
\newtheorem{question}[theorem]{Question}

\DeclareMathOperator{\ev}{ev}

\DeclareMathOperator{\Hom}{Hom}

\newcommand{\RR}{\mathbb{R}}

\newcommand{\NN}{\mathbb{N}}

\newcommand{\Zcal}{\mathcal{Z}}

\newcommand{\f}{\mathrm{f}}
\newcommand{\p}{\mathrm{p}}

\newcommand{\sqC}{\scalebox{0.8}[1.2]{$\sqsubset$}}

\newcommand{\Gm}{\mathbb{G}_{\operatorname{m}}}

\newcommand{\op}[1]{\operatorname{#1}}

\newcommand{\bcd}{\begin{center}\begin{tikzcd}}
\newcommand{\ecd}{\end{tikzcd}\end{center}}

\newcommand{\e}{e}
\newcommand{\Aaff}{\mathbb{A}}

\newcommand{\PP}{\mathbb{P}}
\newcommand{\OO}{\mathcal{O}}
\newcommand{\N}{\mathbb{N}}
\newcommand{\Z}{\mathbb{Z}}
\newcommand{\ZZ}{\mathbb{Z}}

\newcommand{\virt}{\op{vir}}

\newcommand{\Speck}{\operatorname{Spec}\kfield}
\newcommand{\kfield}{\Bbbk}

\newcommand{\Orb}{\mathsf{Orb}}
\newcommand{\LogOrb}{\mathsf{LogOrb}}

\newcommand{\Lcal}{\mathcal{L}}
\newcommand{\Ocal}{\mathcal{O}}

\newcommand{\Acal}{\mathcal{A}}
\newcommand{\Dcal}{\mathcal{D}}
\newcommand{\Xcal}{\mathcal{X}}
\newcommand{\Bcal}{\mathcal{B}}
\newcommand{\Ecal}{\mathcal{E}}
\newcommand{\Ccal}{\mathcal{C}}
\newcommand{\Mfrak}{\mathfrak{M}}

\newcommand{\Pcal}{\mathcal{P}}

\newcommand{\Rder}{\mathbf{R}^\bullet}

\newcommand{\Nup}{\mathsf{N}}
\newcommand{\Mup}{\mathsf{M}}

\newcommand{\Log}{\mathsf{Log}}

\newcommand{\vir}{\text{\rm vir}}

\newcommand{\Trop}{\op{Trop}}

\newcommand{\edit}[1]{\color{black} {#1} \color{black}}
\newcommand{\red}[1]{\color{black} {#1} \color{black}}
\newcommand{\blue}[1]{\color{blue} {#1} \color{black}}

\NewDocumentCommand{\compatibilitydatum}{m m m m m m O{} O{} O{}}{
\begin{equation*} \begin{tikzcd}[ampersand replacement=\&]
  \: \arrow{r} \& {#1} \arrow{r} \arrow{d}{#7} \& {#2} \arrow{r} \arrow{d}{#8} \& {#3} \arrow{r}{[1]} \arrow{d}{#9} \& \: \\
  \: \arrow{r} \& {#4} \arrow{r} \& {#5} \arrow{r} \& {#6} \arrow{r} \& \:
\end{tikzcd} \end{equation*}}

\NewDocumentCommand{\commutingsquare}{m m m m o O{} O{} O{} O{}}{
\begin{equation}\begin{tikzcd}[ampersand replacement=\&] \label{#5}
  #1 \arrow{r}{#6} \arrow{d}{#7} \& #2 \arrow{d}{#8} \\
  #3 \arrow{r}{#9} \& #4
\end{tikzcd}\IfValueTF{#5}{\label{#5}}{} \end{equation}}

\NewDocumentCommand{\Cartesiansquare}{m m m m O{} O{} O{} O{}}{
\begin{equation*}\begin{tikzcd}[ampersand replacement=\&]
  #1 \arrow{r}{#5} \arrow{d}{#6} \arrow[dr, phantom, "\square"] \& #2 \arrow{d}{#7} \\
  #3 \arrow{r}{#8} \& #4
\end{tikzcd} \end{equation*}}

\NewDocumentCommand{\Cartesiansquarelabel}{m m m m m O{} O{} O{} O{}}{
\begin{tikzcd}[ampersand replacement=\&]
  #1 \arrow{r}{#6} \arrow{d}{#7} \arrow[dr, phantom, "\square"] \& #2 \arrow{d}{#8} \\
  #3 \arrow{r}{#9} \& #4
\end{tikzcd}\IfValueTF{#5}{\label{#5}}{}
}

\NewDocumentCommand{\triangleofspaces}{m m m O{} O{} O{}}{
\begin{tikzcd} [ampersand replacement=\&]
#1 \arrow{r}{#4} \arrow[bend right]{rr}{#5} \& #2 \arrow{r}{#6} \& #3
\end{tikzcd}}

\begin{document}
 
\title{Gromov--Witten theory via roots and logarithms}
\author{Luca Battistella, Navid Nabijou and Dhruv Ranganathan}

\begin{abstract} 
Orbifold and logarithmic structures provide independent routes to the virtual enumeration of curves with tangency orders for a simple normal crossings pair $(X|D)$. The theories do not coincide and their relationship has remained mysterious. We prove that the genus zero orbifold theories of multi-root stacks of strata blowups of $(X|D)$ converge to the corresponding logarithmic theory of $(X|D)$. With fixed numerical data, there is an explicit combinatorial criterion that guarantees when a blowup is sufficiently refined for the theories to coincide. The result identifies birational invariance as the crucial property distinguishing the logarithmic and orbifold theories. There are two key ideas in the proof. The first is the construction of a naive Gromov--Witten theory, which serves as an intermediary between roots and logarithms. The second is a smoothing theorem for tropical stable maps; the geometric theorem then follows via virtual intersection theory relative to the universal target. The results import a new set of computational tools into logarithmic Gromov--Witten theory. As an application, we show that the genus zero logarithmic Gromov--Witten theory of a pair is determined by the absolute Gromov--Witten theories of its strata.
\end{abstract}
\maketitle
\setcounter{tocdepth}{1}
\tableofcontents

\section*{Introduction}

\subsection{Overview} Gromov--Witten theory for a pair $(X|D)$ of a projective manifold and a simple normal crossings divisor models the enumerative geometry of curves in $X$ with prescribed tangency along $D$. Two approaches to this problem of significant recent interest proceed via logarithmic structures and via root constructions along the divisor: see~\cite{AbramovichChenLog,ChenLog,GrossSiebertLog} and ~\cite{AbramovichVistoli,Cadman,CadmanChen,AbramovichCadmanWise,TsengYouSNC}. In both, the idealised object is a map of pairs from a smooth curve
\[
(C| p_1+\ldots+p_n)\to (X|D=D_1+\ldots+D_k)
\]
where the tangency of $p_i$ along $D_j$ is a fixed non-negative integer $c_{ij}$. The approaches above yield two paths to a proper moduli space with a virtual fundamental class. In both approaches, the degenerate objects in the moduli space parameterize stable maps with additional structure. 

In the orbifold approach, the map is required to lift to a representable map of stacks, obtained from root constructions along the divisors $D_j$ and the markings $p_i$. The tangency is recorded by the morphisms on isotropy groups, but the orbifold structure along $D$ is not canonical and provides an auxiliary choice -- the rooting parameter(s).

In the logarithmic approach, one demands that the map can be enhanced to a logarithmic one, with tangency orders encoded in the logarithmic structure sheaves. 

The two theories are very different in nature, provide complementary computational tools, and offer different insights into enumerative geometry.
The present paper compares them.
There are two orthogonal sources of complexity:

\begin{enumerate}[(i)]
\item {\bf Source complexity:} the genus of the source curve, where the problem breaks naturally into the cases of genus zero and positive genus;
\item {\bf Target complexity:} the logarithmic rank of the target, where the problem breaks naturally into the cases of logarithmic rank one, where the divisor is smooth, and arbitrary rank.
\end{enumerate}

\noindent
When the logarithmic rank is one there is a complete picture. Once the rooting parameter $r$ is sufficiently high, the theories coincide in genus zero by work of Abramovich, Cadman, Marcus, and Wise, following calculations of Cadman and Chen ~\cite{AbramovichMarcusWise,AbramovichCadmanWise,CadmanChen}. In higher genus, Maulik showed that the orbifold theory depends on the rooting parameter $r$. Work of Janda, Pandharipande, Pixton, and Zvonkine asserts that the contribution to the Gromov--Witten theory that is local to the divisor $D$ is polynomial in $r$, and the $r$-free term is the logarithmic theory~\cite{JPPZ,JPPZ2}. Finally, Tseng and You used the compatibility of these statements with degeneration to deduce that the $r$-free term of the orbifold theory of $(X|D)$ is the logarithmic theory~\cite{TsengYouHigherGenus}.

The present article studies the complementary situation of the genus zero theory in arbitrary logarithmic rank. Examples due to the second and third authors show that even in genus zero the theories do not coincide \cite[Section~1]{MaxContacts}. We offer an explanation: the logarithmic theory is invariant under logarithmic modifications of the boundary divisor \cite{AbramovichWiseBirational}, whereas the orbifold theory is not. We show that the invariant part of the latter is the logarithmic theory. 

\subsection{Main results} Let $(X|D)$ be a simple normal crossings pair with components $D_1,\ldots, D_k$. We fix numerical data $\Lambda$ for the enumerative problem, which is a tuple $(g,\upbeta,n,[c_{ij}])$, 
where $g$ is the genus of the source curve, $\upbeta$ is the curve class, $n$ is the number of marked points, and $c_{ij}$ is the tangency order at the $i$-th marking with respect to $D_j$. The collection of iterated blowups 
\[
(X^\dag|D^\dag)\to (X|D)
\]
along the strata of $D$ forms an inverse system. In Section~\ref{sec: slope subdivision} we introduce a property of such blowups called \textbf{$\Lambda$-sensitivity}, and observe that blowups with this property always exist. Moreover, this condition is stable under taking further blowups. 

For the first result, we fix the genus to be zero. 
Given any iterated blowup $(X^\dag|D^\dag)\to (X|D)$ along logarithmic strata, there is a canonical lifting of the numerical data $\Lambda$ to $(X^\dag|D^\dag)$, obtained by formally computing the strict transform of a stable map: see Section~\ref{sec: naive spaces under blowup}.

\begin{customthm}{X}\label{thm: main-thm}
 Let $(X^\dag|D^\dag)\to (X|D)$ be a $\Lambda$-sensitive blowup. The \lu{genus zero} logarithmic Gromov--Witten invariants of $(X|D)$ with numerical data $\Lambda$ agree with the orbifold Gromov--Witten invariants of $(X^\dag|D^\dag)$ with the same numerical data.
\end{customthm}
The result is proved on the level of virtual cycles: see Section~\ref{sec: introduction proof idea} for details. In particular, the result implies that the multi-root orbifold theories of strata blowups eventually stabilise, which appears to be a non-obvious consequence. 

Combining this with results on Gromov--Witten theory of root stacks \cite{TsengYouCodimOne,ChenDuWang}, gerbes \cite{AJTGerbes1}, projective bundles and blowups~\cite{FanBlowups} we obtain the following \emph{reconstruction principle for logarithmic Gromov--Witten invariants}. 

\begin{customcor}{Y}\label{cor: reconstruction}
The genus zero logarithmic Gromov--Witten theory of $(X|D)$ can be uniquely
reconstructed from the absolute Gromov--Witten theories of all the logarithmic strata of $(X|D)$.
\end{customcor}

In genus zero, this result generalises a well-known theorem from relative Gromov--Witten theory due to Maulik and Pandharipande~\cite[Theorem~2]{MaulikPandharipande}. The authors are not aware of a purely logarithmic argument that yields this result. Although it may be possible to use the degeneration formula to deduce it~\cite{RangExpansions}, the result illustrates the type of consequences that can be extracted from our theorem. 

\subsection{Naive theory} \label{sec: introduction proof idea} The route to Theorem~\ref{thm: main-thm} passes through a new version of logarithmic Gromov--Witten theory, referred to as the naive theory. This appeared in an early form in~\cite[Chapter~3]{NabijouThesis} and \cite{MaxContacts}. \edit{In this paper, we construct the theory in full generality: in arbitrary genus and without positivity assumptions on the divisor.}

Let $\Log_\Lambda(X|D)$ denote the moduli space of logarithmic stable maps to a pair $(X|D)$, and $\mathsf M_\Lambda(X)$ the space of absolute stable maps; the symbol $\Lambda$ records the numerical data. Given $(X|D)$, each divisor component $D_i$ determines a smooth pair $(X|D_i)$, and a logarithmic stable map to $(X|D)$ induces one to each $(X|D_i)$. The geometry of the interior of the space $\Log_\Lambda(X|D)$ suggests an alternative compactification that we call the moduli space of \textbf{naive stable maps}
\[
\mathsf N_\Lambda(X|D) \colonequals \Log_\Lambda(X|D_1)\times_{\mathsf M_\Lambda(X)} \cdots \times_{\mathsf M_\Lambda(X)} \Log_\Lambda(X|D_k).
\]
The fibre product is taken in the category of algebraic stacks, disregarding the logarithmic structures. \red{The following theorem holds for arbitrary genus.} 

\begin{customthm}{Z}\label{thmZ}
The moduli space $\Nup_\Lambda(X|D)$ of naive stable maps is equipped with natural evaluation morphisms to the strata of $(X|D)$, and a virtual fundamental class \red{whose virtual dimension coincides with the virtual dimension of the corresponding space of logarithmic stable maps.}

In genus zero the resulting system of naive Gromov--Witten invariants agrees with the orbifold theory, as follows. There is a diagram of moduli spaces equipped with virtual fundamental classes
\[
\begin{tikzcd}
\Orb_\Lambda(X|D) \ar[dr] & & \mathsf{NLogOrb}_\Lambda(X|D) \ar[dl] \ar[dr] & \\
& \mathsf{NOrb}_\Lambda(X|D) &  &  \Nup_\Lambda(X|D).
\end{tikzcd}
\]
with each arrow identifying virtual fundamental classes via proper pushforward.
\end{customthm}

The spaces appearing in the diagram above will be defined in Section~\ref{sec:rootsandnaive} below. In simple terms, there is no direct morphism from $\Orb_\Lambda(X|D)$ to $\Nup_\Lambda(X|D)$, but a map exists after introducing appropriate root constructions: see for instance~\cite{AbramovichCadmanWise,AbramovichMarcusWise}. The intermediate spaces give an explicit moduli theoretic factorisation of these root constructions. Product formulas in orbifold Gromov--Witten theory play an important role in Theorem~\ref{thmZ}: see~\cite{AJTProducts,BNTY}. The fact that logarithmic Gromov--Witten theory does not satisfy product-type formulas can be seen as an instance of the phenomena that we encounter in this paper: see~\cite{DhruvProduct,MaxContacts}. 

Theorem~\ref{thm: main-thm} can now be phrased as stating that the natural map
\[\Log_\Lambda(X^\dag|D^\dag)\to \Nup_\Lambda(X^\dag|D^\dag)\]
is virtually birational for any $\Lambda$-sensitive blowup $(X^\dag|D^\dag)\to(X|D)$. In Section~\ref{sec: proof} this is proved by reducing to a tropical smoothing problem, which we solve by exploiting the combinatorics of $\Lambda$-sensitive blowups. The solution identifies the main component of the naive moduli space.

\subsection{Context} The enumerative geometry of orbifolds obtained by roots along a simple normal crossings divisor was considered in Cadman's foundational paper~\cite[Section~2]{Cadman}, and its systematic study was initiated by Tseng and You~\cite{TsengYouSNC}. They observed that the theory satisfies a number of remarkable properties, owing to the structural similarities between ordinary and orbifold Gromov--Witten theory. They ask about the relationship with logarithmic Gromov--Witten theory, and the main theorem answers this question in genus zero.

In recent years, a number of results have appeared in the logarithmic geometry literature asserting that natural geometric claims are true \textit{after sufficiently blowing up}, see for instance~\cite{HPS,MaxContacts,RangExpansions,DhruvProduct}. However our results \textit{do not} fit into this general principle. In the present work, the moduli space of stable maps is never directly blown up, but rather only the target is, leading to a much stronger conclusion. In the cited works, once the moduli spaces of curves or of maps is blown up, the canon of intersection theory methods, including localisation, do not apply. These methods remain accessible for us, leading to conclusions such as Corollary~\ref{cor: reconstruction}. In short, we only blowup the target, while earlier results blowup the moduli space. 

The comparison between logarithmic and orbifold Gromov--Witten theory should also be contrasted with those relating different models \textit{within} logarithmic Gromov--Witten theory, including~\cite{AbramovichMarcusWise,AbramovichWiseBirational,RangExpansions}. While the latter statements are essentially uniform, the comparison we obtain is more subtle and only true after appropriate invariance properties are forced. 

The main result has consequences for the types of computational techniques which are available. Computations in logarithmic enumerative geometry have been achieved via tropical correspondences~\cite{MandelRuddat,NishinouSiebert}, degeneration formulas~\cite{AbramovichChenGrossSiebertDegeneration,PuncturedMaps,RangExpansions,YixianWu}, and scattering diagrams~\cite{ArguzGross,GrossPandharipandeSiebert}. These methods rely on tropical geometry in an essential fashion. 

The orbifold theory is equipped with a different set of computational techniques, including torus localisation, Givental formalism, and universal relations such as WDVV and Virasoro constraints: see~\cite{TsengYouSNC}. None of these are available in logarithmic geometry. The difference in computational techniques illustrates the different nature of the two sides, and is a notable feature of Theorem~\ref{thm: main-thm}. As a coarse consequence, the main result, together with virtual localisation, determines the genus zero logarithmic Gromov--Witten theory of a pair $(X|D)$ for $X$ toric and $D$ a subset of the toric boundary in terms of known calculations~\cite[Section~9]{MLiuLocalisation}. In a different way, once the problem has been moved out of logarithmic geometry, quasimap wall-crossing and enumerative mirror symmetry techniques become available~\cite{BN21,CCFK}. It seems plausible that mirror transformations may help to further clarify the relationship between the logarithmic and orbifold theories, see for instance~\cite{Sha23,You22}.


The process of taking the $r$-free term of the higher genus orbifold theory can also be seen as a way of forcing an invariance property in the orbifold theory~\cite{JPPZ,JPPZ2,TsengYouSNC}. Together with the main result here, it suggests that invariance under roots and blowups are precisely what distinguishes the orbifold theory from the logarithmic theory. 

\subsection{Future work} Two directions for theoretical generalisation are available. The first is the higher genus comparison problem. As discussed earlier, if the logarithmic rank is one, the invariants are known to become polynomial in the rooting parameter, with the $r$-free term being the logarithmic theory. The polynomiality result is known to hold for simple normal crossings divisors \cite{TsengYouSNC}: the invariants are polynomial in variables $r_1,\ldots, r_k$ corresponding to the rooting parameters for divisors $D_1,\ldots, D_k$. It is natural to conjecture:
\begin{customconj}{W}\label{conjecture higher genus} For each choice of numerical data $\Lambda$ in arbitrary genus, there is a sufficiently refined blowup $(X^\dag|D^\dag) \to (X|D)$ such that the constant-term orbifold Gromov--Witten invariants of $(X^\dag|D^\dag)$ agree with the logarithmic Gromov--Witten invariants of $(X|D)$.	
\end{customconj}
\noindent The result appears plausible, but it seems likely that new input coming from higher double ramification cycles will be necessary, going beyond~\cite{HPS,HolmesSchwarz,MolchoRanganathan}. 

The second direction is to incorporate negative tangency. In the logarithmic theory this is the subject of~\cite{PuncturedMaps}, while in the orbifold theory, it can be extracted by studying invariants of fixed co-age~\cite{FWY20,FWY21}. The latter papers investigate the comparison between logarithmic and orbifold geometries in logarithmic rank one with negative contact orders and establish a comparison, albeit using a different model for the logarithmic negative contact theory via localisation and graph summation on the normal cone of $D$. Negative contact orders are the key input in the construction of intrinsic mirrors, via degree zero relative quantum cohomology for pairs~\cite{GS19}. The associativity of this ring, which is the main result of loc. cit., is shown via a remarkable calculation. Our results offer the possibility of deducing the associativity from the simpler associativity of orbifold quantum cohomology, which also offers a route to a lograrithmic quantum cohomology for pairs~\cite{TsengYouSNC}.

\subsection{Strategy} A simplified overview is as follows. Given a pair $(X|D)$, each smooth pair $(X|D_i)$ determines a cycle in the space $\Mup(X)$ of absolute maps. The intersection of these cycles is the image of the naive moduli space of maps in $\Mup(X)$. We equip it with a virtual class that is compatible with its presentation as an intersection. The orbifold geometry attached to $(X|D)$ is the fibre product of those attached to $(X|D_i)$, over $X$. The Gromov--Witten cycles for the orbifold theory and, by definition, the naive theory both satisfy a product rule in $\Mup(X)$. Once they are equated, we dispense with the orbifold theory. In the main text, this is refined to work on the naive spaces themselves rather than their images in $\Mup(X)$. 

A map to $X$ lifts to a logarithmic map to $(X|D_i)$ if and only if it satisfies an intersection-theoretic criterion due to Gathmann \cite{GathmannRelative}. However the existence of lifts to all $(X|D_i)$ does not guarantee a lift to $(X|D)$. The obstruction is tropical: a logarithmic map determines a combinatorial type of a tropical map, and when $X$ is replaced with the universal simple normal crossings pair $[\mathbb A^k/\Gm^k]$, the tropical data essentially determines the logarithmic data. However, these tropical data are not compatible with the natural product decomposition of $[\mathbb A^k/\Gm^k]$. Concretely, assembling two rank one tropical types to a single type of rank two is a subtle problem.

However, explicit polyhedral geometry arguments show that the obstructions vanish after a sufficiently fine subdivision of the tropical target. This yields a universal version of the main theorem. Combined with techniques from virtual intersection theory, we deduce the main result. 

\subsection*{Conventions and notation} We work over an algebraically closed field of characteristic zero. We assume familiarity with toroidal, logarithmic and tropical geometry, and the role they play in logarithmic Gromov--Witten theory. An overview is provided in Section~\ref{sec: tropical essentials}.

We will work with various flavours of stable maps throughout the paper. To reduce clutter, we adopt the simplified notation
\[ \Mup_\Lambda(X), \Log_\Lambda(X|D), \Orb_\Lambda(X|D) \]
for the absolute, logarithmic, and orbifold moduli spaces. A number of auxiliary moduli spaces will be introduced at various points and denoted similarly, as will the corresponding spaces of prestable maps to the universal targets.

\subsection*{Acknowledgements} The question addressed here was the subject of a lively discussion with D. Abramovich and R.~Pandharipande in Oberwolfach in 2019, at a workshop organised by D. Abramovich, M. van Garrel, and H.~Ruddat. The question resurfaced in a discussion with M.~Gross and B. Siebert in 2020 following the limit orbifold quantum ring constructed by H.-H. Tseng and F. You~\cite{TsengYouSNC}. We are very grateful to all of them for sharing speculations and objections. 

We have also benefitted from conversations with F. Carocci, T. Graber, A. Kumaran, D. Maulik, S. Molcho, Q. Shafi, I. Smith, M. Talpo, and J. Wise on related topics over the years. We especially thank D.~Abramovich, R.~Pandharipande, H.-H. Tseng,  J. Wise, F. You, and an anonymous referee for comments on a draft of this manuscript. The work as a whole owes an intellectual debt to the line of inquiry opened by C. Cadman and L. Chen~\cite{CadmanChen}. 

\subsection*{Funding} L.B. was supported by the Deutsche Forschungsgemeinschaft (DFG, German Research Foundation) under Germany's Excellence Strategy EXC 2181/1-390900948 (the Heidelberg STRUCTURES Excellence Cluster) and Project-ID 444845124 – TRR 326 (GAUS). N.N. is supported by the Herchel~Smith~Fund. D.R. is supported by EPSRC New Investigator Grant EP/V051830/1. 

\section{Tropical essentials}\label{sec: tropical essentials}

\noindent We provide a review of the crucial aspects of logarithmic and tropical geometry that we require. The treatment strives for brevity, but we provide detailed references.

\subsection{Cone complexes and Artin fans}\label{sec:ConeComplexes} Let $(X|D)$ be a simple normal crossings pair with components $D_1,\ldots, D_k$. \red{A stratum of $(X|D)$ is a connected component of an intersection of some of the $D_j$ (this includes the empty intersection $X$). All strata are smooth. We momentarily make the simplifying assumption that every intersection of $D_j$ is connected.}

Let $\mathcal A$ denote the stack $[\mathbb A^1/\Gm]$. A morphism from $X$ to $\mathcal A$ is precisely the data of a line bundle-section pair, referred to here as a \textbf{generalised Cartier divisor}. Each $D_j$ furnishes a morphism $X\to \mathcal A$. Combining these, the divisor $D$ induces a smooth morphism
\[
X\to \mathcal A^k.
\]
Let $\mathcal A_X$ be the minimal open substack of $\mathcal A^k$ through which the map from $X$ factors. It has one point for each stratum. Let $\mathcal D_X$ be the complement of the dense open point. Consider the sequence of maps
\[
X\to \mathcal A_X\hookrightarrow \mathcal A^k.
\]
Each space is a logarithmic algebraic stack, with logarithmic structure induced by the natural simple normal crossings divisor. Each arrow is strict. The stack $\mathcal A_X$ is called the \textbf{Artin fan} of $X$. See~\cite[Section~2]{AbramovichWiseBirational} for the original construction and~\cite[Section~5]{AbramovichEtAlSkeletons} for further exposition. 

The datum $(X|D)$ determines a cone complex with integral structure, also called the \textbf{tropicalisation of $(X|D)$}. It is denoted $\Sigma_X$, with $D$ suppressed from notation, and is constructed as follows. It has $k$ distinct rays $\rho_1,\ldots, \rho_k$. For each nonempty subset $I\subset \{1,\ldots, k\}$ we write $D_I$ for the intersection of the corresponding $D_i$. If nonempty, the intersection is closed of codimension equal to the size of $I$. For each nonempty $D_I$ the complex $\Sigma_X$ contains a copy of the cone $\sigma_I$, generated by the rays $\rho_i$ for $i$ in $I$. There is a natural identification of $\sigma_I$ as a face of $\sigma_J$ whenever $J$ contains $I$. Define
\[
\Trop(X|D):=\Sigma_X := \varinjlim_{I} \sigma_I
\]
as the colimit under these identifications. It comes with a canonical set of integral points, denoted $\Sigma_X(\N)$. The construction first appeared in a more general form in~\cite[Chapter~II]{KKMS}. An exposition may be found in~\cite[Section~2]{AbramovichEtAlSkeletons} \red{or \cite[Section~2]{CavalieriChanUlirschWise}}. 

\begin{remark}
The partially ordered set of faces of $\Sigma_X$ can be equipped with its Alexandrov topology, and this coincides with the underlying topological space of $\mathcal A_X$.  
\end{remark}

\begin{remark}
If the intersections $D_I$ are not connected, both the Artin fan and cone complex constructions generalise in a natural way. Following the construction above, the map from $X$ to $\mathcal A^k$ has disconnected fibres. By replacing the $\mathcal A_X$ above with a non-separated cover we can guarantee connectedness of the fibres: see~\cite{AbramovichWiseBirational}. The cone complex construction can also be generalised by taking one cone for each connected component in the intersection. 
\end{remark}

\subsection{Tropical curves and tropical maps} \label{sec: tropical curves} The tropical curves in this paper will have genus zero. The \textbf{dual graph} $\Gamma$ of a stable genus zero curve $C$ has vertices in bijection with the irreducible components of $C$ and the set of edges connecting two vertices in bijection with the set of nodes connecting the corresponding components. If $C$ is marked by smooth points $p_1,\ldots, p_n$. This determines a map to the vertex set
\[
m: \{1,\ldots,n\}\to V(\Gamma). 
\]
A \textbf{tropical curve} is an enhancement of $\Gamma$ obtained by assigning a positive real length to each edge
\[
\ell: E(\Gamma)\to\mathbb R_{>0}. 
\]
The triple $(\Gamma,m,\ell)$ determines a metric graph enhancing $\Gamma$ in two steps. First, identify an edge $e$ of $\Gamma$ with an interval of length $\ell(e)$. Second, for each marking $i$, attach a copy of $\RR_{\geq 0}$ to this metric graph by identifying $0 \in \RR_{\geq 0}$ with the vertex $m(i)$. 

The notion of a \textbf{family of tropical curves} is defined as follows. Fix a sharp toric monoid $Q$ and let $\tau$ be the dual cone. Given a marked dual graph, a family of tropical curves over $\tau$ is determined by a map
\begin{equation} \label{eqn: length assignment} \ell: E(\Gamma)\to Q\setminus \{0\}.\end{equation}
Each point of $\tau$ is a homomorphism from $Q$ to $\mathbb R_{\geq 0}$, and by applying this homomorphism to the edge assignment above, each edge $e$ acquires a length. On the interior of $\tau$, this is a tropical curve in the previous sense. Fixing the data \eqref{eqn: length assignment}, there is a cone complex $\sqC$ together with a map
\[
\sqC\to \tau
\]
whose fibre over a point in the interior is a tropical curve as previously discussed. The boundary parameterises tropical curves where a subset of edges have been contracted. The total space $\sqC$ is a cone complex of relative dimension $1$ over $\tau$ and its cones correspond to the faces of the underlying dual graph $\Gamma$. We typically work with this total space $\sqC$, with the labeling of the graph $\Gamma$ implicit. See~\cite[Section~3]{CavalieriChanUlirschWise} for details and generalisations. 

Let $\sqC\to \tau$ be a family of tropical curves and let $\Sigma$ be a cone complex. A \textbf{tropical map from $\sqC$ to $\Sigma$} is a map of cone complexes
\[
\sqC\to \Sigma.
\]
The fibres over $\tau$ can be viewed as tropical curves equipped with a piecewise linear map to $\Sigma$. When $\Sigma= \RR_{\geq 0}$ we will refer to this as a \textbf{piecewise linear function} on $\sqC$. 

\begin{remark}[Slice picture]
In order to speak of a single tropical curve, as opposed to a family, formally one takes the monoid $Q$ above to be $\mathbb N$. In reality, this produces a family of tropical curves $\sqC\to \tau=\mathbb R_{\geq 0}$. In combinatorial arguments it is traditional to slice this map of cone complexes and take the fibre over the point $1$ \lu{in $\tau$}. The family can be uniquely reconstructed from this slice by taking the cone over the tropical curve.
\end{remark}

\subsection{Tropicalisation \lu{of curves}}\label{sec:Tropicalisation} Let $(\Speck,Q)$ denote a closed point equipped with the logarithmic structure associated to a toric monoid $Q$. Let $\tau$ be the dual cone of $Q$. Given a logarithmic curve
\[
\Ccal\to (\Speck,Q)
\]
there is an associated tropical curve $\sqC\to\tau$. Let $\Gamma$ be the dual graph of $\Ccal$. The \'etale local structure of logarithmic curves was determined by Kato~\cite{KatoFCurves}, and associates an element of $Q$ to each edge of $\Gamma$: see~\cite[Section~7.2]{CavalieriChanUlirschWise} for a tropical perspective. The \textbf{tropicalisation of $\Ccal\to (\Speck,Q)$} is the associated family of tropical curves $\sqC\to \tau$. We denote the tropicalisation by $\Trop(\Ccal)$ or $\sqC$.

\begin{remark}
The construction globalises to relate the moduli space of logarithmic curves and the moduli space of tropical curves: see~\cite{CavalieriChanUlirschWise}. We will not require this generality. 
\end{remark}

\subsection{Line bundles and blowups}\label{sec: line-bundles} Two combinatorial dictionaries will be used in the paper. First, if $\Ccal$ is a logarithmic curve over a geometric point $(\Speck,Q)$, then a piecewise linear function on the tropicalisation $\sqC$ determines a generalised Cartier divisor on $\Ccal$. The construction generalises the toric dictionary between toric Cartier divisors and piecewise linear functions on the fan. 

If $\upalpha:\sqC\to \RR_{\geq 0}$ is a piecewise linear function, it can be identified with an element of $H^0(\Ccal,\overline M_{\Ccal}^{\mathrm{gp}})$. The line bundle $\mathcal O_{\Ccal}(\upalpha)$ is the image of $-\upalpha$ in $H^1(\Ccal,\mathcal O_{\Ccal}^\star)$ under the coboundary map induced by
\[
0\to\Ocal_{\Ccal}^\star\to M_{\Ccal}^{\mathrm{gp}}\to \overline M_{\Ccal}^{\mathrm{gp}}\to 0.
\]
The logarithmic structure map $M_{\Ccal}\to\Ocal_{\Ccal}$ furnishes a section of this bundle: see e.g.~\cite[Section~2]{RanganathanSantosParkerWise1}. The resulting generalised Cartier divisor is denoted $(\OO_{\Ccal}(\upalpha),s_\upalpha)$. A piecewise linear function on $\Sigma_X$ similarly determines a divisor on $X$ supported along $D$. 

The second construction is related to birational modifications. A morphism of cone complexes $\Sigma'\to \Sigma_X$ \red{(see Section~\ref{sec:ConeComplexes} and \cite[Chapter II]{KKMS})} is a \textbf{subdivision} if it is a bijection on the underlying sets. The most basic type of subdivision is the {stellar subdivision} of a simplicial cone. Given a simplicial cone $\sigma$ with primitive generators $u_1,\ldots, u_k$ the \textbf{stellar subdivision at $\sigma$} is the unique cone complex that refines $\Sigma_X$ and has precisely one more ray, namely the one spanned by $\sum_{i=1}^k u_i$. A subdivision $\Sigma'\to \Sigma_X$ induces a birational modification $X'\to X$ which  modifies $X$ along the boundary $D$. The stellar subdivision at a cone $\sigma$ is the blowup of $X$ at the smooth stratum dual to $\sigma$. Further combinatorial details can be found in the text on toric geometry~\cite[Section~3.3]{CoxToric}. The generalisation to the context of simple normal crossings pairs is detailed in~\cite{KKMS}. 

A more general version will also be used. If $(\Speck,Q)$ is a geometric point, then a subdivision of $\tau$ gives rise to a logarithmic modification of $\Speck$. We can identify this logarithmic point with a point in the closed orbit of the toric variety with coordinate ring $\kfield[Q]$, equipped with the pullback logarithmic structure. The subdivision of $\tau$ birationally modifies this toric variety. The corresponding modification of $\Speck$ is obtained by pullback along the inclusion. See~\cite[Section~9]{KatoToricSingularities} for a detailed study.

\subsection{Moduli of curves with fixed tropicalisation} \label{sec: moduli of log curves} Let $C$ be genus zero prestable curve over $\Speck$. Choose an identification of the dual graph of $C$ with a labeled graph $\Gamma$ and let $\sigma_\Gamma$ denote the smooth cone $\Hom(\mathbb N^{E(\Gamma)},\RR_{\geq 0})$. \red{Points of $\sigma_\Gamma$ correspond to choices of edge lengths, i.e. metric enhancements of $\Gamma$. This gives rise to a universal family of tropical curves $\sqC_\Gamma \to \sigma_\Gamma$.} Choose a rational ray 
\[
\rho\subset \sigma_\Gamma
\] 
whose interior is contained in the interior of $\sigma_\Gamma$ and let $v$ be the primitive integral generator of $\rho$. We describe the set of all logarithmic curves $\Ccal$ over the standard logarithmic point $(\Speck,\mathbb N)$ such that (i) the underlying curve is $C$, (ii) the dual graph is identified with $\Gamma$, and (iii) the tropicalisation is given by the restriction of $\sqC_\rho\to\rho$ of the universal tropical family to $\rho$.

The moduli space of logarithmic curves with such fixed tropicalisation is described by~\cite[Theorem~4]{CavalieriChanUlirschWise} and the formalism surrounding it. As we have chosen an identification of the dual graph with a fixed copy of $\Gamma$ and only require the case where the underlying base is a geometric point,  the stack theoretic subtleties in loc. cit. do not arise, and the formalism simplifies. 

This moduli space is an algebraic torus, which can be identified with the parameter space of logarithmic morphisms
\[(\Speck,\N)\to(\Speck,\N^{E(\Gamma)})\]
whose tropicalisation is $\rho\hookrightarrow \sigma_\Gamma$. Indeed, a logarithmic lift is given by specifying the image of every generator
\[ (1,e_i)\in\kfield^{\times}\oplus\N^{E(\Gamma)}\]
in $\kfield^{\times}\oplus\N$. The discrete part is determined by the dual of $\rho \hookrightarrow \sigma_\Gamma$ but the continuous part is free.

A geometric view is provided below. A closely related concept in tropical geometry is the initial degeneration: see \cite[Chapters 2-3]{MaclaganSturmfels}. 

\begin{construction}
Given $\rho\subset \sigma_\Gamma$ choose a subdivision of $\sigma_\Gamma$ that includes $\rho$ as a face. Let $\widetilde U$ be the associated logarithmic modification of $(\Speck,\NN^{E(\Gamma)})$. Define
\[ \mathsf M(\rho) \subseteq \widetilde{U} \]
to be the locally closed stratum of codimension $1$ which is dual to $\rho$. 
\end{construction}

Each point of $\mathsf M(\rho)$ gives rise to a logarithmic curve $\Ccal$ over a logarithmic point $(\Speck,\mathbb N)$ enhancing $C$, whose tropicalisation is $\sqC_\rho \to \rho$, and which is compatible with the fixed identification of the dual graph with $\Gamma$. In fact, $\mathsf M(\rho)$ is precisely the moduli space of such data. 

More generally, if we choose any rational cone $\tau\subset \sigma_\Gamma$ with $Q$ the dual monoid of $\tau$, we can ask for logarithmic enhancements of $C$ over $(\Speck,Q)$, again with a fixed identification of the dual graph with $\Gamma$. The moduli space of such enhancements is constructed similarly, 

\subsection{Logarithmic maps to Artin fans} \label{sec: log maps to Artin fans} \red{We will use the following} method for constructing logarithmic maps to Artin fans. Given an Artin fan $(\Acal_X|\Dcal_X)$ and a logarithmic curve $\Ccal \to (\Speck,Q)$, the data of a logarithmic map
\[ \Ccal \to (\Acal_X|\Dcal_X)\]
is equivalent to the data of a map of cone complexes
\begin{equation} \label{eqn: tropical map} \Trop\Ccal =\sqC \to \Sigma= \Trop(\Acal_X|\Dcal_X)\end{equation}
see~\cite[Proposition~2.10]{AbramovichChenGrossSiebertDegeneration}. Here \[ \Trop\Ccal =\sqC \to \tau=\Hom(Q,\RR_{\geq 0})\] is the tropicalisation of $\Ccal \to (\Speck,Q)$.

Typically we start with a scheme-theoretic curve $C$, and proceed by first constructing an appropriate tropical curve $\sqC \to \tau$ and accompanying map $\sqC \to \Sigma$. Once this has been done, it remains to construct any logarithmic enhancement of $C$ which tropicalises to the given tropical curve, which is achieved via the process outlined in Section~\ref{sec: moduli of log curves} above.

\section{Naive--orbifold comparisons}\label{sec:rootsandnaive}

\noindent A naive stable map to a simple normal crossings pair $(X|D)$ is a stable map to $X$ together with a logarithmic lift to each pair $(X|D_i)$. The theory was studied under positivity assumptions in genus zero by the second and third authors~\cite{NabijouThesis,MaxContacts}. The spaces were introduced as part of a framework to reduce questions in logarithmic Gromov--Witten theory to questions in the simpler relative theory. It was later realised that both the geometry of these spaces and their enumerative invariants closely mirror the orbifold theory~\cite[Theorem B]{BNTY}.

We now construct naive Gromov--Witten theory for simple normal crossings pairs without positivity or genus restrictions. In genus zero, we equate the naive invariants with the orbifold invariants, removing the spurious assumptions in~\cite[Theorem B]{BNTY}. In higher genus, we conjecture that the naive invariants coincide with the constant-term orbifold invariants described in~\cite{TsengYouSNC}.

A reader less familiar with orbifold geometry may wish to examine the definition of naive Gromov--Witten theory, take the equality with the orbifold theory as a given, and skip to the later sections where the orbifold theory is dispensed with.

\subsection{Setup} \label{sec: setup} Consider a simple normal crossings pair $(X|D)$. We equip $X$ with the divisorial logarithmic structure corresponding to $D$. Fix numerical data $\Lambda$ for a moduli space of logarithmic stable maps to $(X|D)$. This consists of a genus $g$, a labeling set $\{1,\ldots,n\}$ for the marked points, a curve class $\upbeta \in A_1(X)$, and tangency orders to $D$ at each marked point. The latter are encoded by a list of integral points
\[ 
\upalpha_1,\ldots,\upalpha_n \in \Sigma_X(\N).
\]
Let $D_1,\ldots,D_k$ be the smooth components of the divisor. The tangency conditions are equivalently given by non-negative integers $c_{ij}$ for $1 \leq i \leq k$ and $1 \leq j \leq n$. If $w_i$ is the primitive lattice vector generating the ray corresponding to $D_i$ then for each $j$ we have
\[ \upalpha_j = \sum_{1 \leq i \leq k} c_{ij} w_i. \]

\begin{remark} If every intersection of divisor components $\cap_{i\in I}D_i$ is connected then the data $c_{ij}$ ranging over all $i$ are equivalent to $\upalpha_j$. Connectivity of strata intersections can always be achieved after additional blowups. If the intersections are disconnected, then $\upalpha_j$ contains more information than the $c_{ij}$: it specifies which connected component of the intersection contains $p_j$.\end{remark}

\begin{remark}\label{rem: global-balancing} The numerical data satisfies the \textbf{global balancing condition} with respect to each divisor component if the following equation holds
\begin{equation} \label{eqn: global balancing} \sum_{1\leq j \leq n} c_{ij} = D_i \cdot \upbeta.\end{equation}
If the global balancing condition is not satisfied, the relevant moduli spaces of maps will be empty. We therefore impose it from the outset.\end{remark} 

The data $\Lambda$ induces numerical data for a moduli space of logarithmic stable maps to any pair $(X|E)$ with $E \subseteq D$ (including $E=\emptyset$). The inclusion $E\subseteq D$ determines a map
\[
\Trop(X|D)\to \Trop(X|E)
\]
obtained by projecting away from the coordinates corresponding to divisors not contained in $E$. Given an integral vector $\upalpha_j$ in $\Trop(X|D)$, its image in $\Trop(X|E)$ determines the contact data for the new pair. We abuse notation and denote the compatible numerical data by the same symbol $\Lambda$. Spaces of logarithmic stable maps are denoted
\[\Log_\Lambda(X|D).\]

The numerical data $\Lambda$ induces numerical data for a moduli of twisted\footnote{For readers less familiar with orbifold Gromov--Witten theory, the lectures of Abramovich may prove helpful~\cite{AbramovichLectures}.} stable maps~\cite{AbramovichVistoli,Cadman}. We consider the multi-root stack \cite{Cadman,TsengYouSNC}
\[ \Xcal = X_{D,\vec{r}} \]
where $\vec{r}=(r_1,\ldots,r_k)$ is a vector of rooting parameters. 	The moduli problem considers representable maps from twisted curves, namely $1$-dimensional Deligne--Mumford stacks with stack structure imposed at nodes and markings. Let $\mathcal C$ be a twisted curve, and consider
\[
f:\mathcal C\to \mathcal X.
\]
Let $D_i/r_i$ in $\Xcal$ denote the substack lying over $D_i$ in $X$. The pullback $f^\star \OO_{\Xcal}(D_i/r_i)|_{p_j}$ is a character of the isotropy group at $p_j$. If the order of the stabiliser at $p_j$ is $s_j$, then this character is an integer $k_j$ lying between $0$ and $s_j-1$. The age of $f^\star \OO_{\Xcal}(D_i/r_i)|_{p_j}$ is \[k_j/s_j \in [0,1)\cap \mathbb{Q}.\] The age is related to the tangency order $c_{ij}$ via the formula
\[ 	c_{ij} = r_i \cdot (k_j/s_j)\]
see \cite[(2.1.4)]{CadmanChen}. The same applies for any $E \subseteq D$. Spaces of twisted stable maps to $\Xcal$ will be denoted
\[ \Orb_\Lambda(X|D)\]
with the auxiliary rooting parameters understood to be arbitrarily large.

\red{Finally, for our comparison arguments we will make use of the following hybrid version of the previous two moduli spaces. 
\begin{definition} We denote by $\LogOrb_\Lambda(X|D)$ the moduli space parametrising: \begin{enumerate}[(i)]
\item A logarithmic twisted curve $\Ccal$ in the sense of Olsson \cite{OlssonLogTwisted}. Denoting the coarse moduli map by $\rho\colon\Ccal\to C$, the twisting is encoded in a Kummer extension of logarithmic structures
\[ \rho^*M_C\to M_{\Ccal}.\]
\item A commuting diagram of representable logarithmic morphisms
\bcd[ampersand replacement=\&]
{(\Ccal,M_{\Ccal})}\ar[r]\ar[d] \& {(\Xcal,\Dcal)}\ar[d] \\
{(C,M_C)}\ar[r] \& {(X,D).}
\ecd
\end{enumerate}
It is by now standard to verify that this problem is represented by a proper Deligne--Mumford stack with logarithmic structure. Its virtual fundamental class is discussed in the proof of Theorem \ref{thm: naive is orbifold} below. This moduli space admits forgetful morphisms to both $\Log_\Lambda(X|D)$ and $\Orb_\Lambda(X|D)$.
\end{definition}}

\subsection{Naive invariants}\label{sec: naive space definition} The \textbf{naive space} $\Nup_\Lambda(X|D)$ is defined as the fibre product of logarithmic mapping spaces with respect to the smooth divisor components
\begin{equation} \label{eqn: naive space definition}
\begin{tikzcd}
\Nup_\Lambda(X|D) \ar[r] \ar[d] \ar[rd,phantom,"\square"] & \prod_{i=1}^k \Log_\Lambda(X|D_i) \ar[d] \\
\Mup_\Lambda(X) \ar[r] & \Mup_\Lambda(X)^k.	
\end{tikzcd}
\end{equation}
Crucially, the above fibre product is taken in the category of algebraic stacks, and not in the category of fine and saturated logarithmic algebraic stacks.

We now manipulate this cartesian diagram, in order to highlight the existence and role of the \emph{universal} naive space. This will allow us to endow every (geometric) naive space with a virtual fundamental class, and to define naive Gromov--Witten invariants in all genera.

Let $\Acal=[\Aaff^{\! 1}/\Gm]$ and $\Dcal=[0/\Gm] \subseteq \Acal$ so that $(\Acal|\Dcal)$ is the universal smooth pair. We shall denote by $\Mfrak(\Acal)$ the moduli space of prestable curves endowed with a line bundle-section pair (or generalised Cartier divisor). We shall denote by $\mathfrak{Log}(\Acal|\Dcal)$ the moduli space of logarithmic prestable maps to $(\Acal|\Dcal)$, i.e. logarithmically smooth curves together with a line bundle-section pair such that the section is a logarithmic morphism, where the line bundle is endowed with the sum of the logarithmic structure pulled back from the curve \emph{and} the logarithmic structure induced by the zero section. This space is known to admit a minimal logarithmic structure. \lu{Since
\[ \Trop (\Acal|\Dcal) = \RR_{\geq0},\]
tropicalising a logarithmic prestable map to $(\Acal|\Dcal)$ produces a tropical curve equipped with a piecewise linear function.}

For each smooth divisor component $D_i$ there is a cartesian square \cite[Lemma~A(iv)]{AbramovichMarcusWise}
\begin{equation} \label{eqn: Cartesian square for smooth divisor log space and universal target}
\begin{tikzcd}
\Log_\Lambda(X|D_i) \ar[r] \ar[d] \ar[rd,phantom,"\square"] & \mathfrak{Log}(\Acal|\Dcal) \ar[d] \\
\Mup_\Lambda(X) \ar[r] & \Mfrak(\Acal)	
\end{tikzcd}
\end{equation}
where the morphism $\Mup_\Lambda(X) \to \Mfrak(\Acal)$ performs post-composition with the map $X \to \Acal$ induced by $D_i\subseteq X$. We obtain
\[
\begin{tikzcd}
\Nup_\Lambda(X|D) \ar[r] \ar[d] \ar[rd,phantom,"\square"] & \mathfrak{Log}(\Acal|\Dcal)^k \ar[d] \\
\Mup_\Lambda(X) \ar[r] & \Mfrak(\Acal)^k
\end{tikzcd}
\]
and this square fits into the following commutative diagram
\begin{equation}\label{eqn: big Cartesian diagram}
\begin{tikzcd}
\Nup_\Lambda(X|D) \ar[r] \ar[d] \ar[rd,phantom,"\text{A}"] & \mathfrak{Log}(\Acal|\Dcal)^k \times_{\Mfrak^k} \Mfrak \ar[r] \ar[d] \ar[rd,phantom,"\text{B}"] & \mathfrak{Log}(\Acal|\Dcal)^k \ar[d] \\
\Mup_\Lambda(X) \ar[swap,r,"p"] & \Mfrak(\Acal)^k \times_{\Mfrak^k} \Mfrak \ar[r] \ar[d] \ar[rd,phantom,"\text{C}"] & \Mfrak(\Acal)^k \ar[d] \\
\, & \Mfrak \ar[swap,r,"\Delta"] & \Mfrak^k.
\end{tikzcd}	
\end{equation}
The squares $\text{C}$ and $\text{B+C}$ are cartesian and therefore $\text{B}$ is cartesian. Since $\text{A+B}$ is cartesian it follows that $\text{A}$ is cartesian. Finally, we have $\Mfrak(\Acal)^k \times_{\Mfrak^k} \Mfrak =\Mfrak(\Acal^k)$, and so we obtain the desired square
\begin{equation}\label{eqn: diagram for naive space and universal naive space}
\begin{tikzcd}
\Nup_\Lambda(X|D) \ar[r] \ar[d] \ar[rd,phantom,"\square"] & \mathfrak{Log}(\Acal|\Dcal)^k \times_{\Mfrak^k} \Mfrak \ar[d] \\
\Mup_\Lambda(X) \ar[swap,r,"p"] & \Mfrak(\Acal^k).
\end{tikzcd}	
\end{equation}

\begin{definition}\label{def:universalNaiveSpace}
The \textbf{naive space for the universal target} is the stack
\[ \mathfrak{NLog}(\Acal^k|\partial \Acal^k) \colonequals \mathfrak{Log}(\Acal|\Dcal)^k \times_{\Mfrak^k} \Mfrak.\]
\end{definition}

\begin{remark} \label{rmk: equivalent descriptions of universal naive} A diagram chase shows that the naive space for the universal target is equivalently given by any of the following three fibre products
\[ \mathfrak{Log}(\Acal|\Dcal)^k \times_{\Mfrak^k} \Mfrak = \mathfrak{Log}(\Acal|\Dcal)^k \times_{\Mfrak(\Acal)^k} \Mfrak(\Acal^k) = \Pi_{i=1}^k \mathfrak{Log}(\Acal^k|\Dcal_i) \times_{\Mfrak(\Acal^k)^k} \Mfrak(\Acal^k)\]
where $\Dcal_i \subseteq \Acal^k$ is the $i$th coordinate hyperplane. The third description is closest to the definition of the naive space for geometric targets. The first description is the most useful for proofs.
\end{remark}

We claim that $p \colon \Mup_\Lambda(X) \to \Mfrak(\Acal^k)$ has a relative perfect obstruction theory. If $f$ denotes the universal stable map and $\pi$ is the universal curve, the obstruction theory is the dual of
\begin{equation} \label{eqn: pi star f star Tlog} \Rder \pi_\star f^\star T^{\log}_{X|D}. \end{equation}
When $k=1$ the result is \cite[Proposition~3.1.1]{AbramovichMarcusWise}. The general case has the same proof, noting that when $X$ is equipped with the divisorial logarithmic structure from $D$ we have
\[ \Omega_{X/\operatorname{Log}} = \Omega_{X/\Acal^k} = \Omega^{\log}_{X|D}.\]
The relative obstruction theory \eqref{eqn: pi star f star Tlog} for $p$ induces a virtual pullback $p_{\operatorname{v}}^!$ on Chow groups \cite{ManolachePull}.

The naive space for the universal target is typically not pure-dimensional when $k>1$ (see Remark~\ref{rmk: naive not pure dimensional} below). Nevertheless, its presentation as an intersection of Artin stacks as above endows it with a class of the expected dimension, by Fulton--Kresch intersection theory~\cite{Kresch}. Fixing the numerical data, each factor $\mathfrak{Log}(\Acal|\Dcal)$ is irreducible and carries a fundamental class \cite[Proposition 1.6.1]{AbramovichWiseBirational}. The \textbf{virtual class for the naive space} is therefore defined as
\[ [\Nup_\Lambda(X|D)]^{\virt} \colonequals p^!_{\op{v}}\ \Delta^! \! \left( [\mathfrak{Log}(\Acal|\Dcal)] \times \cdots \times [\mathfrak{Log}(\Acal|\Dcal)]\right).\]
Unusually, this virtual class arises via virtual pullback $p^!_{\op{v}}$ from a reducible base. On the other hand, the composition $p^!_{\op{v}}\ \Delta^!$ of pullbacks can be viewed as endowing $\Nup_\Lambda(X|D)$ with a perfect obstruction theory over the irreducible $\mathfrak{Log}(\Acal|\Dcal)^k$.

\begin{remark}
 \red{We verify that Kresch's hypotheses are satisfied in our case. The stabilisers in $\Mfrak$ are affine and hence by \cite[Proposition 3.5.9]{Kresch} the stack admits a stratification by global quotients. Since it is also smooth, by \cite[p. 530]{Kresch} it admits a refined intersection product. See also \cite[Section~3.1]{BaeSchmitt} for a discussion of these issues in a similar context.}
\end{remark}

We now define numerical naive Gromov--Witten invariants. The space $\Nup_\Lambda(X|D)$ admits a map
\[
\Nup_\Lambda(X|D)\to \mathsf M_\Lambda(X)
\]
and by composing with the evaluation maps $\ev_i:\mathsf M_\Lambda(X)\to X$ we obtain evaluation maps to $X$. Fix a marked point $p_i$ and let $W_i$ be the intersection of divisor components $D_j$ with respect to which $p_i$ has positive contact order. By definition of the naive space, the evaluation morphism at $p_i$ on $\Nup_\Lambda(X|D)$ factors through this intersection. We obtain 
\[
\ev_i: \Nup_\Lambda(X|D)\to W_i.
\]

\begin{definition} \textbf{Naive Gromov--Witten invariants} of $(X|D)$ with numerical data $\Lambda$ are defined as tautological integrals over the naive virtual class
 \[\langle \tau_{a_1}(\gamma_1),\ldots,\tau_{a_n}(\gamma_n)\rangle^{\Nup(X|D)}_\Lambda=\int_{[\Nup_\Lambda(X|D)]^{\vir}}\prod_{i=1}^n\psi_i^{a_i}\ev_i^\star(\gamma_i)\]
  where $W_i\subset X$ is the stratum of $X$ determined by the marked point $p_i$, the class $\gamma_i$ is in the Chow cohomology of $W_i$, and the $\psi$-classes are pulled back from $\Mup_\Lambda(X)$.
\end{definition}

\begin{remark} \label{rmk: naive not pure dimensional} The stack $\mathfrak{NLog}(\Acal^k | \partial \Acal^k)$ is a fibre product of toroidal, i.e. logarithmically smooth, stacks. This fibre product is not transverse along certain boundary strata, and may consist of several irreducible components. The failure of transversality can be detected from the tropicalisations \cite{AbramovichKaru,Molcho}. The fact that there are many irreducible components can be seen from the examples in~\cite[Section~1]{MaxContacts}. These examples consider the space of naive and logarithmic maps to $\mathbb P^2$ relative to two lines. For brevity we call it the \textit{geometric} {naive space}. The examples imply the lack of irreducibility in the universal setting. Indeed, this geometric target is logarithmically convex, and therefore the space of logarithmic maps to this target is smooth over the space of logarithmic maps to the Artin fan. Similarly, by the description of the obstruction theory above, the map from this geometric naive space to the universal one is smooth. The examples in~\cite{MaxContacts} identify points in the geometric naive space of $\mathbb P^2$ that do not lie in the closure of the locus of non-degenerate maps. It follows that there exist naive maps to the Artin fan that cannot be deformed to non-degenerate maps.  

The strata of the geometric and universal naive spaces can be understood as fibre products of strata of the geometric and universal spaces of maps to smooth pairs. The relationship between the strata of these latter spaces can be understood, for example, using~\cite[Section~1.1]{KHNSZ22}. 
\end{remark}

\subsection{Naive-orbifold correspondence} \label{sec: orbifold is naive} We now restrict our attention to \emph{genus zero}. We show that naive invariants and orbifold invariants coincide in this case. This removes the convexity hypothesis of \cite[Theorem B]{BNTY}. 

\begin{theorem} \label{thm: naive is orbifold} There is a diagram of spaces connected by virtually birational morphisms
\[
\begin{tikzcd}
\Orb_\Lambda(X|D) \ar[r] & \mathsf{NOrb}_\Lambda(X|D) & \mathsf{NLogOrb}_\Lambda(X|D) \ar[l] \ar[r] &  \Nup_\Lambda(X|D).
\end{tikzcd}
\]
\end{theorem}
This means that every space in the above diagram is endowed with a virtual fundamental class, and these classes coincide under pushforward along the proper morphisms in the diagram. \red{The auxiliary spaces $\mathsf{NOrb}_\Lambda(X|D)$ and $\mathsf{NLogOrb}_\Lambda(X|D)$ will be defined below.}
\begin{corollary} The naive and orbifold invariants coincide
	 \[\langle \tau_{a_1}(\gamma_1),\ldots,\tau_{a_n}(\gamma_n)\rangle^{\Nup(X|D)}_\Lambda = \langle \tau_{a_1}(\gamma_1),\ldots,\tau_{a_n}(\gamma_n)\rangle^{\mathrm{Orb}(X|D)}_{\Lambda.}\]
\end{corollary}
\begin{proof} This follows immediately from the previous theorem and the projection formula
\end{proof}

\begin{proof}[Proof of Theorem~\ref{thm: naive is orbifold}] We start with the case of a smooth pair. This is a minor modification of the main results of \cite{AbramovichCadmanWise}. For a fixed rooting index $r$ the morphism
\[ \Acal_r \to \Acal \]
will denote the morphism $[\Aaff^1/\Gm] \to [\Aaff^1/\Gm]$ given by $t \mapsto t^r$. This is the \textbf{universal} $r$-th root construction. While $\Acal_r$ is abstractly isomorphic to $\Acal$, we retain the subscript to indicate that the morphism $\Acal_r \to \Acal$ is not the identity.

Consider the moduli stack $\mathfrak{Orb}(\Acal_r)$ parametrising \emph{representable} maps $\Ccal \to \Acal_r$ from a prestable twisted curve, i.e. the data of a generalised Cartier divisor $(\Lcal,s)$ on $\Ccal$ such that, for every point $p$ in $\Ccal$ at which $s(p)=0$, the induced action $\mu_{r(p)} \curvearrowright \Lcal_p$ is faithful, where the integer $r(p)$ is the twisting index. Furthermore, the latter is required to divide $r$, so that the $r$-th tensor power of $\Lcal$ is pulled back from the coarse curve of $\Ccal$.

On the other hand, we define a moduli stack $\mathfrak{LogOrb}(\Acal_r|\Dcal_r)$ whose objects consist of the following data:
\begin{enumerate}[(i)]
\item a logarithmic twisted curve $\Ccal$; 
\item a representable generalised Cartier divisor $(\Lcal,s)$ on $\Ccal$; 
\item \lu{an enhancement of the section to a logarithmic morphism} $s \colon \Ccal \to \Lcal$, where $\Lcal$ is equipped with the sum of the logarithmic structure induced by the pullback of the logarithmic structure on $\Ccal$ and the divisorial logarithmic structure corresponding to the zero section.
\end{enumerate}
It is not hard to see that this moduli space admits a minimal logarithmic structure too. The last condition above ensures that taking the $r$-th power of $(\Lcal,s)$ and forgetting the Kummer extension of the logarithmic structure defines a morphism  $\mathfrak{LogOrb}(\Acal_r|\Dcal_r) \to \mathfrak{Log}(\Acal|\Dcal)$. On the other hand, forgetting the logarithmic structure \emph{in toto} we obtain a morphism $\mathfrak{LogOrb}(\Acal_r|\Dcal_r) \to \mathfrak{Orb}(\Acal_r)$. 

We have thus produced a roof
\begin{equation} \label{eqn: diagram maps universal target smooth divisor}
\begin{tikzcd}
\mathfrak{Orb}(\Acal_r)  & \mathfrak{LogOrb}(\Acal_r|\Dcal_r) \ar[r] \ar[l] & \mathfrak{Log}(\Acal|\Dcal).
\end{tikzcd}	
\end{equation}
As is typical when dealing with an Artin stack which is not of finite type, we must pass to an appropriate open substack which contains the image of the geometric moduli space under consideration. In our context this has been discussed carefully in \cite{AbramovichCadmanWise} where primed superscripts are used to indicate the chosen open substacks. From now on we replace each of the entries in \eqref{eqn: diagram maps universal target smooth divisor} by the appropriate open substack.

With this replacement, minor modifications of the arguments of Abramovich--Cadman--Wise then show that the morphisms in \eqref{eqn: diagram maps universal target smooth divisor} are birational. The passage from universal to geometric targets is then achieved using the following cartesian squares
\[
\begin{tikzcd}
\Log_\Lambda(X|D) \ar[r] \ar[d] \ar[rd,phantom,"\square"] & \Mup_\Lambda(X) \ar[d] & \, \Orb_\Lambda(X|D) \ar[r] \ar[d] \ar[rd,phantom,"\square"]  & \Mup_\Lambda(X) \ar[d] & \,  \mathsf{LogOrb}_\Lambda(X|D) \ar[r] \ar[d] \ar[rd,phantom,"\square"] & \Mup_\Lambda(X) \ar[d] \\
\mathfrak{Log}(\Acal|\Dcal) \ar[r] & \Mfrak(\Acal),  & \mathfrak{Orb}(\Acal_r) \ar[r] & \Mfrak(\Acal), & \, \mathfrak{LogOrb}(\Acal_r|\Dcal_r) \ar[r] & \Mfrak(\Acal).
\end{tikzcd}
\]
The first square is cartesian by \cite[Lemma A (iv)]{AbramovichMarcusWise}. Using the definition of the root stack as $X \times_{\Acal} \Acal_r$ it is easy to see that the second square is cartesian. The third square is taken as the definition of $\mathsf{LogOrb}_\Lambda(X|D)$.

In each case, the virtual fundamental class is the pullback of the fundamental class of the universal moduli space, using the virtual pullback morphism induced by the perfect obstruction theory for the map $\Mup_\Lambda(X) \to \Mfrak(\Acal)$ constructed in the previous section. For the first square this follows by a direct comparison of obstruction theories. For the second square, it follows by the identity
\begin{equation} \label{eqn: log tangent of root stack} p^\star T_{X|D}^{\log} = T_{\Xcal|\Dcal}^{\log} \end{equation}
where $p \colon \Xcal \to X$ is the morphism from the root stack to its coarse moduli space. This holds because $p$ is logarithmically \'etale. Finally for the third square this is taken as the definition of the virtual class on $\mathsf{LogOrb}_\Lambda(X|D)$.

From this, it follows that the birational diagram \eqref{eqn: diagram maps universal target smooth divisor} pulls back to a virtually birational diagram
\[
\begin{tikzcd}
\Orb_\Lambda(X|D) & \mathsf{LogOrb}_\Lambda(X|D) \ar[l] \ar[r] & \Log_\Lambda(X|D).
\end{tikzcd}
\]
This completes the discussion of smooth pairs. Now let $D=D_1+\ldots+D_k$ be simple normal crossings. Each of the three enumerative theories for smooth pairs considered above defines a corresponding naive theory for simple normal crossings pairs. To accompany Definition \ref{def:universalNaiveSpace}, we set forth some naturally defined universal naive spaces
\begin{align*}
 \mathfrak{NOrb}(\Acal^k|\partial \Acal^k) & \colonequals \prod_{i=1}^k \mathfrak{Orb}(\Acal_{r_i}) \times_{\Mfrak^k} \Mfrak,\\
 \mathfrak{NLogOrb}(\Acal^k|\partial \Acal^k) & \colonequals \prod_{i=1}^k \mathfrak{LogOrb}(\Acal_{r_i}|\Dcal_{r_i}) \times_{\Mfrak^k} \Mfrak.
 \end{align*}
While even the \emph{universal} naive spaces may be badly-behaved beyond the base case of a smooth pair, each carries a virtual fundamental class obtained by capping the product of fundamental classes of the factors with the diagonal in the space of curves. Since for smooth pairs the morphisms in \eqref{eqn: diagram maps universal target smooth divisor} preserve the fundamental classes, a diagram chase gives a virtually birational diagram
\[
\begin{tikzcd}
\mathfrak{NOrb}(\Acal^k|\partial \Acal^k)  & \mathfrak{NLogOrb}(\Acal^k| \partial \Acal^k) \ar[r] \ar[l] & \mathfrak{NLog}(\Acal^k|\partial \Acal^k).
\end{tikzcd}
\]
The geometric naive spaces are defined by fibring the universal naive spaces over the map
\[ \Mup_\Lambda(X) \to \Mfrak(\Acal)^k \times_{\Mfrak^k} \Mfrak = \Mfrak(\Acal^k). \]
Their virtual classes are pulled back from the virtual classes of the universal spaces, and the standard comparison of obstruction theories gives a virtually birational diagram
\[
\begin{tikzcd}
\mathsf{NOrb}_\Lambda(X|D)  & \mathsf{NLogOrb}_\Lambda(X|D) \ar[r] \ar[l] & \mathsf{NLog}_\Lambda(X|D) = \Nup_\Lambda(X|D).
\end{tikzcd}
\]
It remains to show that the naive orbifold theory coincides with the orbifold theory. We will show that there is a virtually birational morphism
\[ \Orb_\Lambda(X|D) \to \mathsf{NOrb}_\Lambda(X|D).\]
This is a generalised form of the relative product formula for orbifold invariants established in \cite[Theorem~3.1]{BNTY}. Consider the diagram
\[
\begin{tikzcd}
\Orb_\Lambda(X|D) \ar[r] \ar[d] \ar[rd,phantom,"{\text{A}}"] & \mathsf{NOrb}_\Lambda(X|D) \ar[d] \ar[r] \ar[rd,phantom,"{\text{B}}"] & \Mup_\Lambda(X) \ar[d] \\
\mathfrak{Orb}(\Pi_{i=1}^k \Acal_{r_i}) \ar[r,"\varphi"] & \mathfrak{NOrb}(\Acal^k|\partial \Acal^k) \ar[r] & \Mfrak(\Acal^k).	
\end{tikzcd}
\]
The morphism $\varphi$ is obtained by projecting $\Ccal \to \Pi_{i=1}^k \Acal_{r_i}$ onto each factor of the target and taking the minimal partial coarsening of $\Ccal$ which ensures that the composed map is representable. This partial coarsening is discussed in \cite[Section~9]{AbramovichVistoli} and \cite[Theorem~3.1]{AbramovichOlssonVistoli}, where it is referred to as the ``relative coarse moduli space''.

The square $\text{B}$ is cartesian by definition, and $\text{A+B}$ is clearly cartesian. It follows that $\text{A}$ is cartesian. The vertical arrows in $\text{A}$ are equipped with perfect obstruction theories, which are compatible by the identity \eqref{eqn: log tangent of root stack}. The bases of both vertical arrows in $\text{A}$ are equipped with virtual classes, recalled below. These virtual classes virtually pull back to the virtual classes on the top row of $\text{A}$. By Manolache's formalism \cite[Theorem~4.1]{ManolachePull}, in order to show that $\Orb_\Lambda(X|D) \to \mathsf{NOrb}_\Lambda(X|D)$ is virtually birational it suffices to show that $\varphi$ identifies virtual classes.

We first recall the virtual classes on the bottom row of $\text{A}$. The space $\mathfrak{Orb}(\Pi_{i=1}^k \Acal_{r_i})$ is equipped with a virtual class by a construction of Abramovich--Cadman--Wise. The forgetful morphism to the moduli stack $\Mfrak^{\mathrm{tw}}$ of twisted curves factors as
\[
\mathfrak{Orb}(\Pi_{i=1}^k \Acal_{r_i})\to \mathfrak{Orb}(\mathcal B\Gm^k)\to \Mfrak^{\mathrm{tw}}. 
\]
The first arrow takes a tuple of line bundles and sections and forgets the sections. The second arrow then forgets the line bundles. By \cite[Section~5.2]{AbramovichCadmanWise} the morphism
\[ \mathfrak{Orb}(\Pi_{i=1}^k \Acal_{r_i})\to \mathfrak{Orb}(\mathcal B\Gm^k)\]
is virtually smooth and equips the universal orbifold space $\mathfrak{Orb}(\Pi_{i=1}^k \Acal_{r_i})$ with a virtual class. It coincides with the standard virtual class construction in orbifold Gromov--Witten theory.

On the other hand the virtual class on $ \mathfrak{NOrb}(\Acal^k|\partial \Acal^k)$ is defined via diagonal pullback, analogously to the virtual class of $\mathfrak{NLog}(\Acal^k|\partial \Acal^k)$ defined in Section~\ref{sec: naive space definition}. The orbifold product formula for the target $\Pi_{i=1}^k \Acal_{r_i}$ then shows that the morphism
\[
\varphi: \mathfrak{Orb}(\Pi_{i=1}^k \Acal_{r_i}) \to \mathfrak{NOrb}(\Acal^k|\partial \Acal^k)
\]
identifies virtual classes~\cite{AJTProducts,BehrendProduct}. Note that the product formula is typically stated for schematic or Deligne--Mumford targets and over stable curves, but the arguments apply without change in this setting. 
\end{proof}

\begin{remark} In higher genus, we conjecture (Conjecture~\ref{conjecture higher genus}) that the naive invariants coincide with the constant-term orbifold invariants of \cite{TsengYouSNC}. The missing step is a universal formulation of the higher genus smooth divisor logarithmic-orbifold correspondence \cite{JPPZ2,TsengYouHigherGenus}. We will return to this in future work.\end{remark}

\section{Geometry of the naive space}
\noindent We restrict to genus zero for the rest of the paper. In Theorem~\ref{thm: naive is orbifold} we proved that orbifold and naive invariants coincide. We now begin a study of the more delicate relationship between naive and logarithmic invariants.

\subsection{Numerical criteria} \label{sec: Gathmann criterion} Consider the finite morphism from the naive space to the absolute mapping space
\[ \Nup_\Lambda(X|D) \to \Mup_\Lambda(X). \]
The geometric points of the image have an elementary characterisation, extending an insight originally due to Gathmann. In practice this allows one to explore the geography of the naive space with relative ease.

\noindent \textbf{Gathmann's condition} \cite[Remark 1.4]{GathmannRelative}. For each divisor component $D_i$ consider a connected component $Z$ of the subscheme $f^{-1}(D_i) \subseteq C$. Then:
\begin{enumerate}
\item If $Z$ is zero-dimensional, it must be supported at a marked point $p_j$ and have length $c_{ij}$.\medskip
\item Otherwise, $Z^{\op{red}}=C_0\subseteq C$ is a subcurve. In this case, let $C_1,\ldots,C_r$ denote the irreducible components of $C$ adjacent to $C_0$ and let $q_1,\ldots,q_r$ be the connecting nodes. For $1 \leq j \leq r$ let $m_j \in \N$ be the contact order of $f|_{C_j}$ to $D_i$ at the point $q_j$ in $C_j$. Then we must have
\[ \sum_{p_j\in C_0}c_{ij} -\sum_{1\leq j \leq r} m_j = D_i \cdot f_\star[C_0] = \deg f^\star \mathcal O_X(D_i)|_{C_0}.\]
\end{enumerate}

Gathmann's condition is closely related, though not identical, to the tropical balancing condition, which we now recall.

\noindent \textbf{Balancing condition.} Consider the logarithmic scheme $S = (\Speck,Q)$ with monoid $Q$ and let $f: \Ccal\to X$ be a logarithmic stable map over $S$. The tropicalisation of $\Ccal \to S$ is a family of tropical curves
\[
\sqC\to \sigma
\]
and tropicalising $f$ gives a morphism of cone complexes
\[
\f \colon \sqC\to \Sigma. 
\]
A piecewise linear function on $\sqC$ (or on $\Sigma$) gives rise to a generalised Cartier divisor on $C$ (or on $X$): see Section~\ref{sec: line-bundles}. We can now state the balancing condition~\cite[Section~1.4]{GrossSiebertLog}. 

Let $f: \Ccal\to X$ be as above and fix a piecewise linear function $\upalpha:\Sigma\to \RR$. There are two ways to produce a line bundle on $\Ccal$: (i) take the line bundle $f^\star\mathcal O_X(\upalpha)$, where $\mathcal O_X(\upalpha)$ is the line bundle on $X$ associated to $\upalpha$; (ii) take the line bundle $\OO_\Ccal(\f^\star \upalpha)$ on $C$ associated to the piecewise linear function
\[
\f^\star \upalpha: \sqC\xrightarrow{\f} \Sigma\xrightarrow{\upalpha} \mathbb R.
\]
The logarithmic morphism $f$ encodes an isomorphism $\OO_{\Ccal}(\f^\star \upalpha) \cong f^\star \OO_X(\upalpha)$. The \textbf{balancing condition} is the coarser statement that for every irreducible component $C^\prime \subseteq C$ we have
\begin{equation} \label{eqn: balancing condition definition}
\deg \mathcal O_{\Ccal}(\f^\star \upalpha)|_{C^\prime} = \deg f^\star\mathcal O_X(\upalpha)|_{C^\prime}.
\end{equation}

\begin{lemma} \label{lem: Gathmann criterion} A stable map in $\Mup_\Lambda(X)$ belongs to the image of $\Nup_\Lambda(X|D)$ if and only if it satisfies Gathmann's condition with respect to every divisor component $D_i$.\end{lemma}

\begin{proof} First assume that $D$ is smooth. Fix a stable map $f \colon C \to X$ which admits a logarithmic enhancement. Its tropicalisation satisfies the balancing condition, and we will use this to deduce Gathmann's condition.

Consider the identity piecewise linear function on $\Sigma=\RR_{\geq 0}$. The associated line bundle is $\OO_X(D)$. On the other hand, this piecewise linear function pulls back to the piecewise linear function on $\sqC$ with slopes along the edges determined by the tropical map $\f \colon \sqC \to \Sigma$. By the balancing condition these two line bundles have the same multi-degree. Summing the degrees over all irreducible components of a given $C_0 \subseteq f^{-1}(D)$, we obtain Gathmann's condition.

For the converse, consider the following commuting diagram of moduli spaces
\bcd
& \mathsf{Kim}_\Lambda(X|D) \ar[rd] \ar[ld] & \\
\mathsf{Log}_\Lambda(X|D) \ar[rd] & & \mathsf{Li}_\Lambda(X|D) \ar[ld]\\
& \Mup_\Lambda(X)
\ecd
where $\mathsf{Li}_\Lambda(X|D)$ is Li's space of relative stable maps to expansions \cite{Li1,Li2}, and $\mathsf{Kim}_\Lambda(X|D)$ is Kim's space of logarithmic stable maps to expansions \cite{KimLog}.

Fix a stable map $f \colon C \to X$ satisfying Gathmann's condition. This guarantees the existence of a meromorphic section of $f^\star \OO_X(D)$ on each irreducible component of $f^{-1}(D)$ with specified zero and pole locus. These can be assembled to produce a relative stable map to the expanded target: see \cite[Lemma~5.1.12]{GathmannThesis}. It follows that our given stable map admits a lift to $\mathsf{Li}_\Lambda(X|D)$.

On the other hand, it is a direct consequence of Li's original construction that the morphism $\mathsf{Kim}_\Lambda(X|D) \to \mathsf{Li}_\Lambda(X|D)$ is surjective, the difference being given by saturation: see \cite[Lemma~4.2.2]{AbramovichMarcusWise} and \cite[Section~6]{GrossSiebertLog}. It follows that the given stable map admits a lift to $\mathsf{Kim}_\Lambda(X|D)$ and hence, by commutativity of the diagram, lies in the image of $\Log_\Lambda(X|D) \to \Mup_\Lambda(X)$.

This completes the proof in the smooth divisor case. The general case follows from the definition \eqref{eqn: naive space definition} of the naive space, which implies that the image of $\Nup_\Lambda(X|D) \to \Mup_\Lambda(X)$ is the intersection of the images of $\Log_\Lambda(X|D_i) \to \Mup_\Lambda(X)$ for $1 \leq i \leq k$.
\end{proof}

\subsection{The data structure of combinatorial types}
A logarithmic map $\Ccal\to (X|D)$ over the logarithmic point $(\Speck,Q)$ produces a family of tropical maps $\sqC\to\Sigma$ over the dual cone $\tau$. The combinatorics of this tropical map enhances the dual graph of the source curve with data recording the vanishing orders of the divisor equations at generic points of components, and at marked and singular points. This data is referred to as the \textbf{combinatorial type} \cite[Section~2.5]{AbramovichChenGrossSiebertDegeneration}.

As we now explain, a naive map also has a combinatorial type, which can be packaged in the same data structure. Recall that the cones of a cone complex form a poset, where $\sigma_1\preceq\sigma_2$ if $\sigma_1$ is a face of $\sigma_2$.
\begin{definition} A \textbf{combinatorial type} of tropical stable map to $\Sigma$ consists of the following data:
\begin{enumerate}[(i)]
\item A finite graph $\Gamma$ consisting of vertices, finite edges and semi-infinite legs. Each leg $\ell$ in $L(\Gamma)$ carries a marking label $j_{\ell} \in \{1,\ldots,n\}$, and each vertex $v$ in $V(\Gamma)$ carries genus and degree labels
\[ g_v \in \NN, \quad \upbeta_v \in A_1(X).\]
\item Cones of $\Sigma$ associated to every vertex, edge and leg of $\Gamma$
\begin{align*}
v \rightsquigarrow \sigma_v, e \rightsquigarrow \sigma_e, \ell \rightsquigarrow \sigma_\ell
\end{align*}
such that if $v \preceq e$ then $\sigma_v \preceq \sigma_e$ for every vertex $v$ and adjacent edge or leg $e$.
\item For every oriented edge or leg $\vec{e}$ an associated slope $m_{\vec{e}} \in N_{\sigma_e}$ satisfying $m_{\vec{e}} = -m_{\cev{e}}$.
\end{enumerate}
\end{definition}

\subsection{Combinatorial type of a naive map} \label{sec: combinatorial type of naive map} We will associate a combinatorial type to each naive map. It will record the dual graph of the curve, its interaction with the target stratification, and the tangency orders at the special points.

\begin{remark}
The combinatorial type of a naive map fits into the same data structure as the combinatorial type of a naive map. However, there will exist naive types that do not arise from logarithmic maps, or equivalently, naive types that are not the underlying combinatorial type of a tropical map. The latter are said to be \textbf{smoothable}. The detection of smoothable types among naive types is the main issue at play in Section~\ref{sec: proof}.
\end{remark}

Fix a genus zero naive stable map $f \colon C \to (X|D)$.

\noindent \textbf{Step I: labeled graph.}  The graph $\Gamma$ is the dual graph of $C$, with vertices corresponding to irreducible components, finite edges to nodes and semi-infinite legs to markings. The leg and vertex labels are determined in the obvious way. Since we work in genus zero, all $g_v=0$.

\noindent \textbf{Step II: associated cones.} Let $C_v \subseteq C$ be an irreducible component. Let $X_v \subseteq X$ be the unique minimal stratum of $(X|D)$ through which the map $C_v \to X$ factors. This corresponds to a cone $\sigma_v \preceq \Sigma$ which is defined to be the cone associated to $v$. The cones $\sigma_e$ and $\sigma_\ell$ are defined similarly.

\noindent \textbf{Step III: slopes along legs.} The slopes along the legs are part of the input data, corresponding to the specified tangencies at the markings. In the notation of Section~\ref{sec: setup}, we have
\[ m_{\vec{\ell}} = \upalpha_{j(\ell)} \in \Sigma(\mathbb N)\]
where $\vec{\ell}$ is oriented to point away from its supporting vertex. Gathmann's condition implies that
\[ m_{\vec{\ell}} \in \sigma_{\ell}(\mathbb N) \subseteq N_{\sigma_{\ell}}.\]

\noindent \textbf{Step IV: slopes along edges.} The slopes $m_{\vec{e}}$ along the finite edges are determined by formally imposing the balancing condition at every vertex. Fix a divisor component $D_i$. For each vertex $v$ of $\Gamma$ the balancing equation gives
\begin{equation}\label{eqn: balancing} \sum_{p_j \in C_v} c_{ij} + \sum_{v\preceq e} (m_{\vec{e}})_i = D_i \cdot \upbeta_v \end{equation}
where every edge is oriented to point away from $v$. The $(m_{\vec{e}})_i$ are unknown integers to be solved for. The first term is a sum over markings $p_j$ in $C_v$, or equivalently legs $\ell_j$ containing $v$. The tangency order $c_{ij}$ is equal to the component of $\upalpha_j$ in the direction corresponding to $D_i$.

The system of equations~\eqref{eqn: balancing} above will appear repeatedly in this section, and will be referred to as the \textbf{balancing equations}. 

\begin{lemma} \label{lem: formal balancing unique solution} The system of balancing equations \eqref{eqn: balancing} has a unique solution in the $(m_{\vec{e}})_i$.
\end{lemma}
\begin{proof} There are $\#V(\Gamma)$ balancing equations \eqref{eqn: balancing}, and a single linear dependency; summing all the equations, the $(m_{\vec{e}})_i$ cancel and we are left with the global balancing constraint
\begin{equation} \label{eqn: global balancing constraint in comb type section} \sum_{1 \leq j \leq n} c_{ij} = D_i \cdot \upbeta \end{equation}
which holds by assumption: see Remark~\ref{rem: global-balancing}. We obtain $\# V(\Gamma)-1$ linearly independent equations. Since $\Gamma$ has genus zero, $\# V(\Gamma)-1 = \#E(\Gamma)$ which coincides with the number of unknowns. Hence there is a unique solution for the $(m_{\vec{e}})_i$.

The uniqueness may be deduced constructively. Select a root vertex and orient the edges of $\Gamma$ to point towards the root. Place a level structure on $\Gamma$ which respects the orientation, i.e. an assignments of heights $h(v) \in \RR$ such that $h(v_1) < h(v_2)$ if there is an oriented edge $\vec{e}$ connecting $v_1$ to $v_2$. Iterate along $\Gamma$ towards the root, from the lowest level to the highest. At each step, the vertex encountered has several incoming edges whose $(m_{\vec{e}})_i$ have previously been determined, and a single outgoing edge whose $(m_{\vec{e}})_i$ is currently unknown. The latter is determined uniquely by solving the balancing equation at the current vertex. This process terminates at the root vertex, where the global balancing constraint \eqref{eqn: global balancing constraint in comb type section} ensures that the balancing equation at the root vertex is satisfied.
\end{proof}

By the above lemma, equation \eqref{eqn: balancing} has a solution for each $D_i$, and this produces integers $(m_{\vec{e}})_i$ for all $\vec{e}$ and $i$. We now use the assumption that $f \colon C\to (X|D)$ is a naive map to assemble these integers into slope vectors $m_{\vec{e}} \in N_{\sigma_e}$. 

Fix an oriented edge $\vec{e}$ and let $w_{i_1},\ldots,w_{i_c}$ be the primitive generators of the cone $\sigma_e$, corresponding to divisor components $D_{i_1},\ldots,D_{i_c}$. 

\begin{lemma} \label{lem: tangencies vanish in irrelevant directions} We have $(m_{\vec{e}})_i=0$ for $i \not\in \{i_1,\ldots,i_c\}$.\end{lemma}
\noindent Once this is proved, we can define the slope along $\vec{e}$ as
\begin{equation} m_{\vec{e}} \colonequals  (m_{\vec{e}})_{i_1} w_{i_1} + \ldots + (m_{\vec{e}})_{i_c} w_{i_c} \in N_{\sigma_e}  \end{equation}
which completes the construction.

It remains to establish Lemma~\ref{lem: tangencies vanish in irrelevant directions}. This follows from the more general Lemma~\ref{lem: geometric and formal tangencies coincide} below, which asserts that the formal tangency orders $(m_{\vec{e}})_i$ constructed above coincide with the geometric tangency orders, provided the latter can be defined. This makes crucial use Gathmann's condition. 
\begin{definition} A node $q$ in $C$ has \textbf{well-defined geometric tangency} with respect to $D_i$ if there exists a neighbourhood $U$ of $q$ such that $f(U) \not\subseteq D_i$.\end{definition}
Equivalently, a point $q$ in $C$ has well-defined geometric tangency if $f(C_1) \not\subseteq D_i$ or $f(C_2) \not\subseteq D_i$ where $C_1,C_2 \subseteq C$ are the irreducible components adjacent to $q$.
\begin{definition}
For every node $q$ with well-defined geometric tangency, let $e$ in $E(\Gamma)$ be the corresponding edge and denote the \textbf{geometric tangency} by
\[ (m_{\vec{e}})_i^\prime \in \NN.\]
Here if $f(C_1) \not\subseteq D_i$ and $f(C_2) \subseteq D_i$ then we orient $e$ to point from $v_1$ to $v_2$. If neither component is mapped into $D_i$ then Gathmann's condition, combined with the global balancing constraint \eqref{eqn: global balancing}, implies that $f(q_e)$ does not lie in $D_i$ and so the geometric tangency is zero.
\end{definition}

\begin{lemma}\label{lem: geometric and formal tangencies coincide} If $q$ in $C$ is a node with well-defined geometric tangency and $e$ in $E(\Gamma)$ is the corresponding edge, then
\[ (m_{\vec{e}})_i = (m_{\vec{e}})^\prime_i.\]	
\end{lemma}
\begin{proof} Consider the collection of geometric tangencies $(m_{\vec{e}})^\prime_i$ and formal tangencies $(m_{\vec{\ell}})_i$ for each leg $\ell$. Note that the latter satisfy the global balancing constraint \eqref{eqn: global balancing constraint in comb type section}. For all remaining edges, introduce unknowns $(m_{\vec{e}})^\prime_i$ and consider the resulting system of balancing equations, as in the preceding discussion.

The system coincides with the one in \eqref{eqn: balancing}, with a minor difference: the $(m_{\vec{e}})^\prime_i$ corresponding to nodes with well-defined tangency are specified at the start. As such there are fewer unknowns.

There are also more linear dependencies. In fact, the system of equations decomposes into a sum of systems of equations with independent sets of variables, one for each maximal subcurve $C_0 \subseteq f^{-1}(D_i)$. The same argument we have used previously, applied to each of these independent systems, guarantees a unique solution for the $(m_{\vec{e}})^\prime_i$. Thus we have two solutions for the balancing equations: the $(m_{\vec{e}})_i$ and the $(m_{\vec{e}})_i^\prime$. By uniqueness they coincide, so the formal tangency orders $(m_{\vec{e}})_i$ coincide with the geometric tangency orders $(m_{\vec{e}})_i^\prime$ whenever the latter are defined.\end{proof}

\begin{proof}[Proof of Lemma~\ref{lem: tangencies vanish in irrelevant directions}]
Fix an index $i \not\in \{ i_1,\ldots,i_c\}$. By the definition of $\sigma_e$ we have $f(q_e) \not\in D_i$. Let $C_1,C_2$ be the irreducible components of $C$ adjacent to $q_e$. Certainly either $f(C_1)$ or $f(C_2)$ is not contained in $D_i$, so $q_e$ has well-defined tangency. Since $f(q_e)$ does not lie in $D_i$ this geometric tangency is zero, so $(m_{\vec{e}})_i=0$ as claimed.
\end{proof}

This completes the construction of the combinatorial type associated to a naive stable map.

\begin{remark} \label{rmk: formal tangency equals geometric tangency} Given a naive stable map $f \colon C \to (X|D)$ the preceding construction assigns formal tangency vectors $m_{\vec{e}}$ along the finite edges of the dual graph. Lemma~\ref{lem: geometric and formal tangencies coincide} can be restated as follows.

Consider a vertex $v$ of $\Gamma$ and an edge or leg $e$ of $\Gamma$ containing $v$, and orient $e$ to point away from $v$. Suppose that $D_i$ is a divisor component which contains the stratum $X_e=X_{\sigma_e}$ but does not contain the larger stratum $X_v=X_{\sigma_v}$. Then the component $(m_{\vec{e}})_i \in \ZZ$ of the vector $m_{\vec{e}} \in N_{\sigma_e}$ must agree with the geometric tangency order of $f|_{C_v}$ to $D_i$ at the special point $q_e$ in $C_v$. The geometric tangency order is well-defined because $f(C_v) \not\subseteq D_i$. Note in particular that we must have $(m_{\vec{e}})_i > 0$.
\end{remark}

\begin{example} \label{example: non-realisable type} Let $X=\PP^2$ with $D=L_1+L_2$ the sum of two lines. Consider a degree~$4$ stable map with the following dual graph
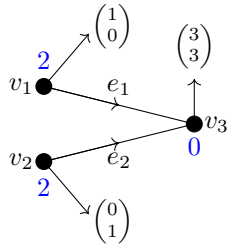
\begin{center}
\begin{tikzpicture}

	\draw[fill=black] (2,0) circle[radius=3pt];
	\draw (2,0) node[right]{\small$v_3$};
	\draw[blue] (2,-0.05) node[below]{\small$0$};
	\draw[->] (2,0) -- (2,0.6);
	\draw (2,0.6) node[above]{\tiny$\begin{pmatrix} 3 \\ 3 \end{pmatrix}$};
	
	\draw[fill=black] (0,0.5) circle[radius=3pt];	
	\draw (0,0.5) node[left]{\small$v_1$};
	\draw[blue] (0,0.6) node[above]{\small$2$};
	\draw [->] (0,0.5) -- (0.6,1.2);
	\draw (0.5,1.3) node[right]{\tiny$\begin{pmatrix} 1 \\ 0 \\ \end{pmatrix}$};
	
	\draw (0,0.5) -- (2,0);
	\draw[->] (0,0.5) -- (1,0.25);
	\draw (1,0.25) node[above]{\small$e_1$};
	
	\draw[fill=black] (0,-0.5) circle[radius=3pt];
	\draw (0,-0.5) node[left]{\small$v_2$};
	\draw[blue] (0,-0.6) node[below]{\small$2$};
	\draw[->] (0,-0.5) -- (0.6,-1.2);
	\draw (0.5,-1.3) node[right]{\tiny$\begin{pmatrix} 0 \\ 1 \end{pmatrix}$};
	
	\draw (0,-0.5) -- (2,0);
	\draw[->] (0,-0.5) -- (1,-0.25);
	\draw (1,-0.25) node[below]{\small$e_2$};
\end{tikzpicture}
\end{center}
The degree of the stable map on each component is indicated in blue, and the tangency orders at the marked points are indicated by column vectors. Orienting the edges as above, the system of balancing equations at the vertices is
\begin{alignat*}{2}
v_1 \colon \, & && m_{\vec{e}_1} + \left(\begin{smallmatrix} 1 \\ 0 \end{smallmatrix}\right) = \left(\begin{smallmatrix} 2 \\ 2 \end{smallmatrix}\right) \\
v_2 \colon \, & && m_{\vec{e}_2} + \left(\begin{smallmatrix} 0 \\ 1 \end{smallmatrix}\right) = \left(\begin{smallmatrix} 2 \\ 2 \end{smallmatrix}\right) \\
v_3 \colon \, & -m_{\vec{e}_1} - && m_{\vec{e}_2} + \left(\begin{smallmatrix} 3 \\ 3 \end{smallmatrix}\right) = \left(\begin{smallmatrix} 0 \\ 0 \end{smallmatrix}\right)
\end{alignat*}
to which there is a unique solution, namely
\[ m_{\vec{e}_1} = \left(\begin{smallmatrix} 1 \\ 2 \end{smallmatrix}\right), \quad m_{\vec{e}_2} = \left(\begin{smallmatrix} 2 \\ 1 \end{smallmatrix}\right). \]
\end{example}

\subsection{Naive types vs. logarithmic types}\label{sec: realisability} We briefly address the combinatorial question of how the naive and logarithmic types differ. Although this discussion is not logically necessary, it provides context for the arguments of Section~\ref{sec: proof}. 

\begin{definition}\label{def: smoothable} A combinatorial type of tropical stable map is \textbf{smoothable} if there exist assignments of edge lengths metrising the graph $\Gamma$
\[ \ell_e \in \RR_{>0} \]
and a continuous map $\f \colon \Gamma \to \Sigma$ such that:
\begin{enumerate}[(i)]
\item every vertex, edge and leg $p$ of $\Gamma$ is mapped into the interior of the cone $\sigma_p$.
\item on every edge or leg $e$, the map $\f|_e \colon e \to \sigma_e$ is affine with slope equal to $m_{\vec{e}}$. 
\end{enumerate}\end{definition}

The positivity of the $\ell_e$ and the fact that faces of $\Gamma$ must map to the interiors of their cones is crucial. Without it, every combinatorial type is smoothable for trivial reasons. The following rephrases~\cite[Remark~1.21]{GrossSiebertLog} and the surrounding discussion. The proof is omitted.

\begin{proposition}
Let $f: C\to (X|D)$ be a logarithmic map over a geometric point. The combinatorial type of $f$ is smoothable. 
\end{proposition}

The next example shows that naive combinatorial types are not always smoothable.

\begin{example}
Consider Example~\ref{example: non-realisable type} above, and suppose that the associated cones are
\[ \sigma_{v_1}=\sigma_{v_2}=0, \quad \sigma_{v_3}=\RR^2_{\geq 0}.\]
There is then no smoothing, because a distance-preserving map $\f$ would impose the following simultaneous equations on the edge lengths
\begin{align*} \ell(e_1) & = 2 \ell(e_2) \\
2 \ell(e_1) & = \ell(e_2)	
\end{align*}
which are only satisfied when $\ell(e_1)=\ell(e_2)=0$.
\end{example}

The following purely combinatorial question arises.
\begin{question}
Which combinatorial types of naive stable maps are smoothable?
\end{question}

In fact, the answer to this combinatorial question dictates the entire comparison between logarithmic and orbifold Gromov--Witten theory. In Section~\ref{sec: proof} we prove that after performing suitably many blowups of the target, every combinatorial type of naive stable map is smoothable.

\subsection{Naive maps and target blowups} \label{sec: naive spaces under blowup} 
We examine the behaviour of naive maps under target blowups and describe the associated transformation on combinatorial types.

 A \textbf{simple blowup} of a simple normal crossings pair is the blowup along a stratum. The reduced preimage of the boundary divisor is again simple normal crossings. We examine an iteration of such simple blowups
\[ \pi \colon (X^\dag|D^\dag) \to (X|D). \]
It is often useful to note that the strata of a simple normal crossings pair are regularly embedded, with canonically split normal bundles, and that $\pi$ is a local complete intersection morphism. Since $\pi$ is a logarithmic modification, it corresponds to a subdivision of cone complexes
\[ \Sigma^\dag \to \Sigma.\]

Consider numerical data $\Lambda$ for a moduli space of stable maps to $(X|D)$. The data $\Lambda$ lifts uniquely to numerical data for maps to $(X^\dag|D^\dag)$, as follows. The genus and marking set are unchanged. For $1 \leq i \leq n$ we use the identification of integral lattice points of the tropicalisation to lift the tangency data at the marked point $p_i$
\[ \upalpha_i \in \Sigma(\mathbb N) = \Sigma^\dag(\mathbb N).\]
It remains to lift the curve curve class. We show how to do this in the case where $\pi$ is a single blowup; the general case follows by induction. Suppose then that we have
\[ \pi \colon X^\dag = \text{Bl}_Z X \to X \]
a blowup \lu{along} $Z \subseteq X$ with exceptional divisor $E \subseteq X^\dag$. Consider the projection
\[ \mathrm{p}_E \colon \Sigma^\dag \to \RR_{\geq 0} \]
corresponding to the forgetful logarithmic morphism $(X^\dag|D^\dag) \to (X^\dag|E)$. Alternatively, this is the piecewise linear function corresponding to the Cartier divisor $E$. Let
\[ d_E \colonequals \sum_{i=1}^n \mathrm{p}_E(\upalpha_i) \in \NN.\]
Let $L \in A_1(X^\dag)$ be the class of a line in a fibre of the projective bundle $E \to Z$, and note that $E \cdot L=-1$. We define
\[ \upbeta^\dag \colonequals \pi^\star \upbeta - d_E L.\]
The global balancing constraint \eqref{eqn: global balancing} is clearly satisfied, and $\pi_\star \upbeta^\dag = \upbeta$. These define the lifted numerical data $\Lambda$ for $(X^\dag|D^\dag)$.

\begin{remark} Suppose that each $\upalpha_i$ in $\Sigma(\mathbb N)$ lies on a ray of $\Sigma$. Geometrically, this means that each marking carries tangency to at most one irreducible component of the divisor, and we refer to this as the \textbf{disjoint case}. Then we have $\upbeta^\dag = \pi^\star \upbeta$. We may always reduce to the disjoint case by performing finitely many \red{initial} blowups of the target.\end{remark}

The next task is to explain how combinatorial types behave \red{under blowups.} A geometric point $s$ of $\Nup_\Lambda(X|D)$ has a combinatorial type. By the discussion of Section~\ref{sec: combinatorial type of naive map} this type is determined by (i) the numerical data $\Lambda$ and (ii) the image of $s$ in $\Mup_\Lambda(X)$. 

\begin{proposition}\label{prop: naive-under-blowup}
Consider the diagram of mapping spaces with compatible numerical data
\[
\begin{tikzcd}
\Nup_\Lambda(X^\dagger|D^\dagger)\arrow[swap]{d}{\theta^\dagger}&\Nup_\Lambda(X|D)\arrow{d}{\theta}\\
\Mup_\Lambda(X^\dagger)\arrow[swap]{r}{\varphi} & \Mup_\Lambda(X). 
\end{tikzcd}
\]
Let $s$ be a geometric point of $\Nup_\Lambda(X^\dagger|D^\dagger)$. Then the set
\[
\theta^{-1}(\varphi\circ\theta^\dagger(s))
\]
is nonempty, finite, and every element has the same combinatorial type. 
\end{proposition}

The following lemma will be key to the proof. Let $C$ be an $n$-pointed curve, and fix tangency orders $c_1,\ldots, c_n$ in $\NN$ for a map to a smooth pair. A morphism $f: C\to \Acal$ \textbf{admits a logarithmic lift} if $f$ is the underlying map of a logarithmic map with these tangency orders. 

\begin{lemma}\label{lem: tensor-business}
Let $C$ be an $n$-marked nodal curve over $\Speck$, and let
\[
f_1,f_2:C\to \Acal
\]
be two morphisms with $(L_1,s_1),(L_2,s_2)$ the associated generalised Cartier divisors. Let 
\[
h:C\to \Acal
\]
be the morphism determined by $(L_1\otimes L_2,s_1\otimes s_2)$. Suppose $f_1$ and $f_2$ admit logarithmic lifts, where $f_i$ has contact order $c_{ij}$ at the marked point $p_j$. Then $h$ also admits a logarithmic lift with contact order $c_{1j}+c_{2j}$ at $p_j$. 
\end{lemma}

\begin{proof}
Since $f_1$ and $f_2$ admit logarithmic enhancements we may consider the combinatorial types of tropical maps to $\RR_{\geq 0}$ of these enhancements. We first predict the combinatorial type of $h$. The graph $\Gamma$ is the dual graph of $C$. If $p$ is a vertex, edge, or leg of $\Gamma$ define the target cone by the following rule: $\sigma_p=0$ if and only if $p$ is assigned the cone $0$ in the types of both $f_1$ and $f_2$. If $\vec e$ is a directed edge of $\Gamma$ we let $m_{\vec e}^i$ denote the slope along this directed edge in the combinatorial type of $f_i$, viewed as an element in $N_{\sigma_e}$. We then define
\[ m_{\vec{e}} \colonequals m_{\vec{e}}^1 + m_{\vec{e}}^2 \in N_{\sigma_e}.\]
The parallel definition is taken at the legs, as demanded by the lemma. This gives a well-defined combinatorial type for a map to $\RR_{\geq 0}$. 

Since all types of tropical maps to $\RR_{\geq 0}$ are smoothable, we can find (i) a logarithmic structure on $\Speck$ with monoid $Q$ and dual cone $\tau$, (ii) an enhancement of $C$ to a logarithmic curve $\Ccal$, and (iii) a map from the associated family of tropical curves
\[
\begin{tikzcd}
\sqC\arrow{d}\arrow{r} & \RR_{\geq 0}\\
\tau.
\end{tikzcd}
\]
As discussed in Section~\ref{sec: log maps to Artin fans} we deduce that there is a logarithmic morphism
\[
\widetilde h: \Ccal \to \Acal.
\]
Now, the map $\widetilde h$ satisfies all the demands placed on $h$ in the statement of the lemma, so we need only show that $h$ is isomorphic to the underlying morphism of $\widetilde h$. Each of $h$ and $\widetilde h$ produces a generalised Cartier divisor, and we must show that these are isomorphic.

At each vertex $v$ of $\Gamma$, the line bundle associated to  $f_i$ is given by a divisor, specifically the sum of distinguished points of $C_v$ weighted by the slopes of the associated piecewise linear function: see e.g. \cite[Proposition~2.4.1]{RanganathanSantosParkerWise1}. The line bundle associated to their tensor product is the line bundle associated to the sum of these divisors. It follows immediately from the definition that the line bundle associated to $h$ and $\tilde h$ are isomorphic. It is also easy to see that the isomorphism preserves the sections. The proof is thus complete.
\end{proof}

\subsubsection*{Proof of Proposition~\ref{prop: naive-under-blowup}}

It follows from Section~\ref{sec: combinatorial type of naive map} that the fibres of both vertical morphisms in the proposition parameterise maps with the same combinatorial type. The finiteness follows from~\cite[Corollary~1.2]{WiseUniqueness}. The main work is to demonstrate nonemptiness.

Consider a single blowup along a smooth stratum
\[
(X^\dagger|D^\dagger)\to (X|D).
\]
Let $D_1,\ldots,D_k$ be the divisor components of $D$. Let $\tilde D_1,\ldots, \tilde D_k$ denote their strict transforms and $E$ the exceptional. The pullback of $D_i$ is either $\tilde D_i$ or $\tilde D_i+E$. 

Let $C\to (X^\dagger|D^\dagger)$ be a naive stable map and consider the map
\[
f: C\to X^\dagger\to X
\]
obtained by composition, without stabilising. We verify that for each $D_i$, $f$ admits a logarithmic lift $f_i \colon \Ccal_i \to (X|D_i)$. By~\eqref{eqn: diagram for naive space and universal naive space} we may equivalently check that the generalised Cartier divisor 
\[
(f^\star\mathcal O_X(D_i),s_{D_i})
\]
admits a logarithmic lift, when viewed as a map to $\Acal$. If $D_i$ pulls back to $\tilde D_i$ then there is nothing to prove. If it pulls back to $\tilde{D}_i+E$ then the result follows from Lemma~\ref{lem: tensor-business} \red{by taking $L_1 = f^\star \OO_X(\tilde{D}_i)$ and $L_2 = f^\star \OO_X(E)$ so that $L_1 \otimes L_2 = f^\star \OO_X(D_i)$.} Since the stabilisation of a logarithmic stable map carries a natural logarithmic structure \red{\cite[Appendix B]{AbramovichMarcusWise}} the claim follows. 
\qed

\begin{remark} Proposition~\ref{prop: naive-under-blowup} gives a multi-valued map between naive spaces, which is a weak form of functoriality. In case it may be useful in future work, we record here that the proof shows that the arrow $\varphi$ restricts to a morphism between the loci of stable maps that admit a naive lift. In fact, if we use spaces of fine but not-necessarily saturated logarithmic maps, as \red{introduced by Chen \cite[Lemma 3.7.4]{ChenLog} and advocated by} Wise \cite{WiseUniqueness}, then the multi-valued map above becomes a morphism. 
\end{remark}

\begin{construction}[Pushforward of types] \label{construction: pushforward of types}
As a consequence of the above proposition, although we have not produced a morphism between naive spaces, there is a morphism between their images in the mapping spaces. Let
\[
\pi: (X^\dagger|D^\dagger)\to (X|D)
\]
be an iterated blowup along strata. Given a combinatorial type for a naive map to $(X^\dagger|D^\dagger)$, the induced type of the map to $(X|D)$ is obtained as follows. The blowup is given by a subdivision
\[ \p \colon \Sigma^\dag \to \Sigma.\]
The labeled source graph $\Gamma$ is replaced by its stabilisation, obtained by deleting bivalent or univalent vertices whose curve classes are exceptional. For every vertex, edge or leg $p$ of $\Gamma$ we replace the cone $\sigma_p \preceq \Sigma^\dag$ by the minimal cone $\overline{\sigma}_p \preceq \Sigma$ containing the image of $\sigma_p$. Finally for every oriented edge or leg $\vec{e}$ we replace $m_{\vec{e}}$ by its image
\[ \p(m_{\vec{e}}) \in N_{\overline{\sigma}_e}.\]
\end{construction}

\section{Slope-sensitivity} \label{sec: slope subdivision}

\noindent The present section contains the main polyhedral arguments in the paper. We introduce the notion of a slope-sensitive subdivision. The term ``slope'' here refers to the edge directions in the combinatorial types of naive stable maps. Once the numerical data is fixed, these form a finite set of directions in the cone complex $\Sigma$. After subdividing $\Sigma$, we show that the combinatorial types of naive maps become tightly constrained. In Section~\ref{sec: proof} we use this to establish the main theorem. 

\subsection{Rank two} The \textbf{rank} of a normal crossings pair $(X|D)$ is the dimension of the tropicalisation. The pair $(X|D)$ is of rank one if and only if $D$ is smooth, and is of rank two if and only if it contains no triple intersection points. We first define slope-sensitivity in the rank two case.

Consider the space of naive stable maps to $(X|D)$ of class $\Lambda$. The set of combinatorial types of such maps is finite. For each combinatorial type, consider the set of slopes $m_{\vec{e}}$ in $N_{\sigma_e}$ which additionally lie in the positive quadrant $\sigma_e \cap N_{\sigma_e} \subseteq N_{\sigma_e}$. Enumerate these slopes, running over all combinatorial types, as follows
\[
m_1, \ldots, m_h \in \Sigma(\mathbb N).
\]
A smooth subdivision of $\Sigma$ which includes amongst its cones the rays generated by $m_1,\ldots,m_h$ is referred to as \textbf{$\Lambda$-sensitive}. We will also use the term \textbf{slope-sensitive} when $\Lambda$ is unambiguous. Such subdivisions always exist, and form a cofinal subset in the inverse system of smooth subdivisions. We refer to the induced logarithmic modification as a $\Lambda$-sensitive modification.

\red{ The set of slopes $m_{\vec{e}}$ appearing in combinatorial types of class $\Lambda$ is unchanged by subdividing the target, as this operation does not alter the support of $\Sigma$ (see also Section~\ref{sec: naive spaces under blowup}). This ensures that $\Lambda$-sensitive subdivisions always exist. Since $\Sigma^\dag$ is required to be smooth, there is typically more than one minimal choice.
 
 Note that while $m_{\vec{e}} \in \Sigma(\N)$ does not change, the minimal cone $\sigma_{e} \preceq \Sigma$ containing the edge $e$ can. This will alter the expression of $m_{\vec{e}}$ in terms of the ray generators of $\sigma_e$. This fact is important for motivating the results of Section~\ref{sec:small_jumping}.
}

\subsection{Higher rank} Write $D=D_1+\ldots+D_k$. For a subset of size two
\[ I=\{i_1,i_2\} \subseteq \{1,\ldots,k\} \]
we let $D_I=D_{i_1}+D_{i_2} \subseteq D$ be the corresponding divisor. Note that $\Sigma_I = \Trop(X|D_I)$ is either $1$-dimensional or $2$-dimensional, depending on whether $D_{i_1} \cap D_{i_2}$ is empty or nonempty. There is a logarithmic morphism
\[ (X|D) \to (X|D_I) \]
inducing a tropical projection
\[ p_I \colon \Sigma \to \Sigma_I.\]
For each $I$ we choose a $\Lambda$-sensitive subdivision $\Sigma_I^\dag$ of the rank two tropical target $\Sigma_I$. This induces a subdivision
\[ p_I^{-1}(\Sigma_I^\dag) \to \Sigma. \]
A smooth subdivision $\Sigma^\dag \to \Sigma$ is then declared to be \textbf{$\Lambda$-sensitive} if it refines all the subdivisions $p_I^{-1}(\Sigma_I^\dag)$ simultaneously. Again, $\Lambda$-sensitive subdivisions always exist and are cofinal in the system of smooth subdivisions. 

\subsection{Iterated blowups} The system of iterated stellar subdivisions of a smooth fan is cofinal in the system of all refinements of that fan: see for instance~\cite[Remark~2.13]{BoteroDivisors}. We may therefore assume, by passing to a further refinement, that $\Sigma^\dag \to \Sigma$ is an iterated stellar subdivision, so that $X^\dag \to X$ is an iterated blowup in smooth logarithmic strata. Summarising, we have:

\red{\begin{lemma} Fix numerical data $\Lambda$. Then there exists an iterated strata blowup
\[ (X^\dag | D^\dag) \to (X|D) \]
such that $(X^\dag | D^\dag)$ is $\Lambda$-sensitive with respect to the lifted numerical data $\Lambda$. Moreover any smooth logarithmic modification of $(X^\dag | D^\dag)$ is also $\Lambda$-sensitive.
\end{lemma}}

\subsection{Properties: small jumping and slope negativity}\label{sec:small_jumping}  Fix a $\Lambda$-sensitive modification
\[ (X^\dag | D^\dag) \to (X|D).\]
By Section~\ref{sec: naive spaces under blowup}, the numerical data lifts uniquely to data for maps to $(X^\dag | D^\dag)$, also denoted $\Lambda$. Fix a naive map $C \to (X^\dag|D^\dag)$ with numerical data $\Lambda$ and consider the associated combinatorial type of tropical stable map to $\Sigma^\dag$ as in Section~\ref{sec: combinatorial type of naive map}.
	
\begin{proposition} \label{lem: once is enough} Given an oriented edge $\vec{e}$ of the dual graph, label the generators of the associated cone by
\[ \sigma_e = \op{Cone}(u_1,\ldots,u_c)\preceq \Sigma^\dag \]
and express the slope along $\vec{e}$ as
\[ m_{\vec{e}} = \sum_{i=1}^c a_i u_i \in N_{\sigma_e}. \]
Then for all $1 \leq i \neq j \leq c$, the pair $(a_i,a_j)$ is not contained in $\Z^2_{>0}$ nor in $\Z^2_{<0}$.
\end{proposition}

\begin{proof} After relabeling the coordinates, we suppose for a contradiction that $(a_1,a_2) \in \ZZ^2_{>0}$. Consider the minimal cone $\overline{\sigma}_e \preceq \Sigma$ containing the image of $\sigma_e \preceq \Sigma^\dag$ under the subdivision $\Sigma^\dag \to \Sigma$. Label its generators as
\[ \overline{\sigma}_e = \op{Cone}(v_1,\ldots,v_s).\]
We claim that there exist indices $1 \leq i \neq j \leq s$ such that the image of the composition
\[ \op{Cone}(u_1,u_2) \to \sigma_e  \to \overline{\sigma}_e \to \op{Cone}(v_i,v_j) \]
is two-dimensional. Otherwise, $u_1$ and $u_2$ must map to the same ray in every $\op{Cone}(v_i,v_j)$ and must therefore map to the same ray in $\op{Cone}(v_1,\ldots,v_s)$. Since $\sigma_e \to \overline{\sigma}_e$ is injective, it follows that $u_1$ and $u_2$ are proportional, a contradiction. By relabeling again, assume that the induced map
\[ \op{Cone}(u_1,u_2) \to \op{Cone}(v_1,v_2) \]
has two-dimensional image. Consider the commutative diagram of cones
\[
\begin{tikzcd}
\sigma_e = \op{Cone}(u_1,\ldots,u_c) \ar[r,"\mathrm{p}"] \ar[d,"\mathrm{f}"] & \op{Cone}(u_1,u_2) \ar[d,"\mathrm{g}"] \\
\overline{\sigma}_e=\op{Cone}(v_1,\ldots,v_s) \ar[r,"\mathrm{q}"] & \op{Cone}(v_1,v_2).
\end{tikzcd}
\]
The naive map to $(X^\dag|D^\dag)$ induces a naive map to $(X|D)$: see Construction~\ref{construction: pushforward of types}. The combinatorial type of this map has slope along $\vec{e}$ given by $\mathrm{f}(m_{\vec{e}})$ in $N_{\overline{\sigma}_e}$. By commutativity of the square we have
\[ \mathrm{q}\circ\mathrm{f}(m_{\vec{e}}) = \mathrm{g}\circ\mathrm{p}(m_{\vec{e}}).\]
Now $\mathrm{p}(m_{\vec{e}}) = (a_1,a_2) \in \ZZ^2_{>0}$ by assumption. Since $\mathrm{g}$ has two-dimensional image it maps the interior of the source cone to the interior of the target cone, so
\[ \mathrm{g}\circ\mathrm{p}(m_{\vec{e}}) \in \ZZ^2_{>0}.\]
Every $\Lambda$-sensitive subdivision of $(X|D_{v_1}+D_{v_2})$ therefore includes the ray generated by $\mathrm{g}\circ\mathrm{p}(m_{\vec{e}})$. Consequently every $\Lambda$-sensitive subdivision of $(X|D)$ must contain the preimage of this ray in $\sigma_e = \op{Cone}(u_1,\ldots,u_c)$ as a union of cones. But this preimage is a hyperplane in the interior of $\sigma_e$, so we have a contradiction.\end{proof}

\begin{corollary}[Small jumping] \label{lem: small jumping} For every flag consisting of a vertex $v$ incident to an edge $e$ in the graph $\Gamma$ we have
\[ 0 \leq \dim \sigma_e - \dim \sigma_v \leq 1. \]
\end{corollary}
\begin{proof} We have $\sigma_e \succeq \sigma_v$ and so $\dim\sigma_e - \dim\sigma_v \geq 0$ is immediate. For the other inequality, assume for a contradiction that $\dim\sigma_e \geq \dim\sigma_v + 2$. Order the generators of the cones so that
\begin{align*}
\sigma_e & = \op{Cone}(w_1,\ldots,w_s,u_1,\ldots,u_c), \\
\sigma_v & = \op{Cone}(u_1,\ldots,u_c),
\end{align*}
with $s \geq 2$. Orient $e$ to point away from $v$ and express the slope $m_{\vec{e}} \in N_{\sigma_e}$ in terms of the above basis
\[ m_{\vec{e}} =  \sum_{i=1}^s b_i w_i + \sum_{i=1}^c a_i u_i.\]
Using the equality of formal and geometric tangency orders explained in Lemma~\ref{lem: geometric and formal tangencies coincide} and Remark~\ref{rmk: formal tangency equals geometric tangency}, we have $b_i > 0$ for $1 \leq i \leq s$. Since $s \geq 2$ this contradicts Proposition~\ref{lem: once is enough}.\end{proof}

\begin{corollary}[Slope negativity] \label{lem: slope negativity} Consider a flag of a vertex $v$ and incident edge $e$, such that $\dim \sigma_e = \dim \sigma_v + 1$. Order the generators of the cones so that
\begin{align*} \sigma_e & = \op{Cone}(u_0,u_1,\ldots,u_c), \\
\sigma_v & = \op{Cone}(u_1,\ldots,u_c).
\end{align*}
Orient $e$ to point away from $v$ and express the slope $m_{\vec{e}}$ in terms of the above basis
\[ m_{\vec{e}} = a_0 u_0 + \sum_{i=1}^c a_i u_i.\]
Then $a_0>0$ and $a_i \leq 0$ for all $1 \leq i \leq c$.
\end{corollary}
\begin{proof} The fact that $a_0>0$ follows from Lemma~\ref{lem: geometric and formal tangencies coincide} and Remark~\ref{rmk: formal tangency equals geometric tangency}. The second statement then follows from Proposition~\ref{lem: once is enough}.
\end{proof}

\section{Logarithmic--naive comparisons}\label{sec: proof}
\noindent We establish the main comparison result. To ease notation, we replace the initial simple normal crossings pair with a slope-sensitive modification, which we continue to denote $(X|D)$.

We will equate the logarithmic and orbifold theories of $(X|D)$. By Theorem~\ref{thm: naive is orbifold}, it suffices to equate the logarithmic and naive theories. The main Theorem~\ref{thm: main-thm} is thus reduced to the following:
\begin{theorem} \label{thm: log is naive} \red{Let $\Lambda$ be numerical data for a space of genus zero logarithmic stable maps to $(X|D)$. Replace $(X|D)$ with a $\Lambda$-sensitive modification.} 
Then the natural morphism
\[ \iota \colon \Log_\Lambda(X|D) \to \Nup_\Lambda(X|D) \]
is virtually birational, i.e.
\[ \iota_\star [\Log_\Lambda(X|D)]^{\virt} = [\Nup_\Lambda(X|D)]^{\virt}. \] 	
\end{theorem}

The proof is given in two parts. In Section~\ref{sec: tropical lift} we establish the theorem on the tropical level. We show that every combinatorial type of naive stable map is smoothable, thus identifying the logarithmic and naive tropical moduli spaces. In Section~\ref{sec: reduction to tropical lift} we then use the geometry of the universal target to conclude the algebro-geometric Theorem~\ref{thm: log is naive} from its tropical counterpart.

\subsection{Tropical lifting} \label{sec: tropical lift} Let $\Acal_X$ be the Artin fan of the pair $(X|D)$ with divisor $\Dcal_X$. As in Section~\ref{sec: naive space definition} let $\Acal$ be the Artin fan of a smooth pair, with universal divisor $\Dcal=\Bcal\Gm$. For each $1 \leq i \leq k$ there is a well-defined projection $$(\Acal_X|\Dcal_X) \to (\Acal|\Dcal)$$ 
obtained by equipping the stack $\Acal_X$ with the Cartier divisor associated to each component of $\Dcal_X$. In Section~\ref{sec: naive space definition}, the universal naive space was defined as
\[ \mathfrak{NLog}(\Acal^k|\partial \Acal^k)\colonequals \mathfrak{Log}(\Acal|\Dcal)^k \times_{\Mfrak^k} \Mfrak.\]
Consider a naive map $C \to (\Acal_X|\Dcal_X)$ over a geometric point. By definition, we have logarithmic maps $\Ccal_i \to (\Acal|\Dcal)$ for each $1\leq i \leq k$. Each gives a family of tropical maps
\[
\begin{tikzcd}
\sqC_i \ar[d] \ar[r] & \RR_{\geq 0} \\
\tau_i
\end{tikzcd}
\]
where $\tau_i$ is the corresponding tropical moduli space of edge lengths and target offsets, dual to the base monoid of $\Ccal_i$. Setting $\sigma_\Gamma = \Pi_{e \in E(\Gamma)} \RR_{\geq 0}$ where $\Gamma$ is the dual graph of $C$, there is a universal family of tropical curves (see Sections~\ref{sec: tropical curves} and \ref{sec: moduli of log curves})
\[ \sqC_\Gamma \to \sigma_\Gamma.\] 
For each $1 \leq i \leq k$ there is a natural morphism $\tau_i \to \sigma_\Gamma$ and $\sqC_i$ is the pullback of $\sqC_\Gamma$.

As discussed in Section~\ref{sec: combinatorial type of naive map}, the naive map produces a combinatorial type of tropical stable map to $\Sigma$. The following tropical lifting theorem shows that this combinatorial type is smoothable in the sense of Definition~\ref{def: smoothable}. The theorem only holds after passing to a slope-sensitive subdivision: see the counterexample given in Section~\ref{sec: realisability}.

\begin{theorem}\label{thm: tropical lift} The fibre product of the $\tau_i$ over $\sigma_\Gamma$ contains a one-dimensional cone
\[ \rho \subseteq \prod_{i=1}^k \tau_i \times_{\sigma_\Gamma^k} \sigma_\Gamma \]
such that the interior of $\rho$ \red{is contained in the interior of the fibre product cone.} This satisfies the following properties. Let $\sqC \to \rho$ denote the pullback of $\sqC_\Gamma \to \sigma_\Gamma$. Then there is a tropical map
\[
\begin{tikzcd}
\sqC \ar[d] \ar[r] & \Sigma \\
\rho
\end{tikzcd}
\]
whose combinatorial type agrees with the combinatorial type of the given naive stable map. Moreover for each $1 \leq i \leq k$ the composition $\sqC \to \Sigma \to \RR_{\geq 0}$, where $\Sigma \to \RR_{\geq 0}$ is the projection corresponding to the logarithmic morphism $(X|D) \to (X|D_i)$, coincides with the pullback of 
\[
\begin{tikzcd}
\sqC_i \ar[d] \ar[r] & \RR_{\geq 0} \\
\tau_i
\end{tikzcd}
\]
along the morphism $\rho \to \tau_i$.\end{theorem}

\begin{proof} It suffices to construct a single tropical map to $\Sigma$ with all tropical parameters --- edge lengths and target positions --- non-zero; we then obtain a one-parameter family by scaling. We will find strictly positive real edge lengths $\ell_e$ giving rise to a tropical curve $\sqC$, and a tropical map
\[ \f \colon \sqC \to \Sigma \]
such that for each vertex $v$ the image $\f(v)$ is contained in $\op{Int}\sigma_v$.

We construct $\f$ by iterating along $\Gamma$. The first step consists of choosing a vertex $v$ of $\Gamma$ and letting $\f(v)$ be an arbitrary interior point of $\sigma_v$. For the induction step, we traverse along an edge $e$ connecting $v_1$ to $v_2$. We assume that $\f(v_1)$ in $\op{Int}\sigma_{v_1}$ has already been constructed, and show that there exists a choice $\ell_e > 0$ of edge length such that if we define $\f(v_2) := \f(v_1) + \ell_e m_{\vec{e}}$ then $\f(v_2)$ is contained in $\op{Int} \sigma_{v_2}$. Since $\Gamma$ has genus zero, there are no loops to potentially cause inconsistencies.

For the induction step, there are several cases to consider. The most difficult is moving from a larger to a smaller cone, i.e.
\begin{equation*} \sigma_{v_1} = \sigma_e \succeq \sigma_{v_2}.\end{equation*}
By Corollary~\ref{lem: small jumping} we know that $\sigma_{v_2}$ is a facet of $\sigma_{v_1}$. Label the generators as follows
\begin{align*}
\sigma_{v_1} &= \op{Cone}(u_0,u_1,\ldots,u_c), \\
\sigma_{v_2} &= \op{Cone}(u_1,\ldots,u_c).
\end{align*}
Then $\f(v_1)$ in $\op{Int}\sigma_{v_1}$ can be written as
\begin{equation*} \f(v_1) = \mu_0 u_0 + \sum_{i=1}^c \mu_i u_i. \end{equation*}
with all $\mu_i > 0$. We orient the edge $e$ from $v_2$ to $v_1$ and write the slope $m_{\vec{e}}$ as
\begin{equation*} m_{\vec{e}} = a_0 u_0 + \sum_{i=1}^c a_i u_i.\end{equation*}
We have $a_0>0$ by Lemma~\ref{lem: geometric and formal tangencies coincide} and Remark~\ref{rmk: formal tangency equals geometric tangency}. By Corollary~\ref{lem: slope negativity} it follows that $a_i \leq 0$ for all $1 \leq i \leq c$. We define
\begin{equation*} \ell_e \colonequals \mu_0/a_0 \in \RR_{>0}.\end{equation*}
This gives
\begin{equation*} \f(v_2) = \f(v_1) - \ell_e m_{\vec{e}} = \sum_{i=1}^c (\mu_i - \mu_0 a_i/a_0) u_i .\end{equation*}
The negative sign arises from the choice of orientation for $e$. Since $\mu_i > 0$ and $\mu_0 a_i/a_0 \leq 0$ we conclude that $\f(v_2)$ lies in $\op{Int}\sigma_{v_2}$ as required.

This deals with the difficult case. Now suppose that $\sigma_{v_1} \preceq \sigma_{v_2}$. In this case we are moving from a smaller to a larger cone, and any sufficiently small $\ell_e > 0$ will suffice. Note that this includes the case $\sigma_{v_1}=\sigma_{v_2}$. The final case is where $\sigma_{v_1}$ and $\sigma_{v_2}$ are distinct facets of $\sigma_e$
\[ \sigma_{v_1} \preceq \sigma_e \succeq \sigma_{v_2}.\]
We subdivide the edge $e$ by introducing an additional bivalent vertex $v_0$, and assign $\sigma_{v_0}=\sigma_e$. We then apply the previous two cases to construct the map. The vertex $v_0$ may now be deleted, and we conclude the result.
\end{proof}

\subsection{Geometric lifting} \label{sec: reduction to tropical lift} We now use the geometry of the universal target to reduce Theorem~\ref{thm: log is naive} to the tropical lifting Theorem~\ref{thm: tropical lift} proved above. Consider the following diagram
\[
\begin{tikzcd}
\Log_\Lambda(X|D) \ar[r] \ar[d,"p"] \ar[rd,phantom,"\text{A}"] & \Nup_\Lambda(X|D) \ar[r] \ar[d,"q"] \ar[rd,phantom,"\text{B}"] & \Mup_\Lambda(X) \ar[d] \\
\mathfrak{Log}(\Acal_X|\Dcal_X) \ar[r,"\psi"] & \mathfrak{NLog}(\Acal^k|\partial \Acal^k) \ar[r]  & \Mfrak(\Acal^k). 
\end{tikzcd}
\]
We note that $\Acal_X$ is the open substack of $\Acal^k$ corresponding to the subcomplex inclusion\footnote{Strictly speaking, this is only a subcomplex inclusion if all intersections $\cap_{i \in I} D_i$ are connected. In general $\Sigma \to \RR_{\geq 0}^k$ is only a finite cover of a subcomplex, but in either case the map $\Acal_X \to \Acal^k$ is open which is all we require. Alternatively, we can force the component intersections to be connected by performing further blowups.}
\[ \Sigma \subseteq \RR_{\geq 0}^k.\]
It follows that the outer square $\text{A+B}$ is cartesian. On the other hand we saw in \eqref{eqn: diagram for naive space and universal naive space} that the square $\text{B}$ is cartesian. We conclude that $\text{A}$ is cartesian.

The morphisms $p$ and $q$ are equipped with compatible perfect obstruction theories: see Section~\ref{sec: naive space definition}. By the Costello--Herr--Wise comparison theorem~\cite{Costello,HerrWiseCostello}, Theorem~\ref{thm: log is naive} follows if we show that $\psi$ is \lu{proper and birational}~\cite[Remark~2.6]{HerrWiseCostello}.

\lu{The arguments of \cite{WiseUniqueness} apply equally when the target is a logarithmic algebraic stack. They imply that $\mathfrak{Log}(\Acal_X|\Dcal_X) \to \Mfrak(\Acal^k)$ is finite over its image, hence proper. From Remark~\ref{rmk: equivalent descriptions of universal naive} we see that the morphism $\mathfrak{NLog}(\Acal^k|\partial \Acal^k) \to \Mfrak(\Acal^k)$ is the base change of a proper morphism, hence \emph{a fortiori} separated. We conclude that $\psi$ is proper.
}

\red{To prove Theorem~\ref{thm: log is naive}, it remains to show that $\psi$ is birational.} The source of $\psi$ is logarithmically smooth, irreducible and equal to the closure of the locus where the logarithmic structure is trivial. More explicitly, this is the locus of non-degenerate maps: the source curve is smooth and its generic point maps to the interior of $\mathcal A_X$. Moreover, $\psi$ induces an isomorphism of this locus onto the same locus of non-degenerate maps in the codomain. As such, if we prove that $\psi$ is surjective then it follows that $\psi$ is birational. Therefore Theorem~\ref{thm: log is naive}, and hence Theorem~\ref{thm: main-thm}, reduces to the following.

\begin{theorem} The following morphism is surjective
	\[ \psi \colon \mathfrak{Log}(\Acal_X | \Dcal_X) \to \mathfrak{NLog}(\Acal^k|\partial \Acal^k).\]
\end{theorem}
\begin{proof} Choose a geometric point of the codomain. This consists of the following data:
\begin{enumerate}[$\bullet$]
\item $C$ a fixed scheme-theoretic source curve;
\item $\Ccal_i \to (\Speck,Q_i)$ a logarithmic enhancement of $C$, one for each $1 \leq i \leq k$;
\item $\sqC_i \to \RR_{\geq 0}$ a morphism of cone complexes for $1 \leq i \leq k$, where $\sqC_i \to \tau_i$ is the family of tropical curves obtained by tropicalising $\Ccal_i \to (\Speck,Q_i)$. Here $\tau_i=\Hom(Q_i,\RR_{\geq 0})$.
\end{enumerate}
The maps $\sqC_i \to \RR_{\geq 0}$ correspond to logarithmic maps $\Ccal_i \to (\Acal|\Dcal)$. Forgetting logarithmic structures, we obtain a map $C \to \Acal^k$ which factors through $\Acal_X$. We construct a logarithmic enhancement of this by constructing a strongly transverse map to an expansion of $\Acal_X$, following~\cite{RangExpansions}. \red{Strong transversality means that the map is disjoint from codimension-$2$ strata and only intersects codimension-$1$ strata at special points, with balanced tangency orders at nodes.}

The expansion is produced from the previously-constructed tropical lift. In Theorem~\ref{thm: tropical lift} we proved the existence of a one-dimensional cone
\[ \rho \subseteq \prod_{i=1}^k \tau_i \times_{\sigma_\Gamma^k} \sigma_\Gamma \]
and compatible tropical map $\sqC \to \Sigma$ metrised by $\rho$. Consider the graph map
\[ \sqC\to \Sigma \times \rho. \]
Toroidal semistable reduction \cite{AbramovichKaru} produces a logarithmic modification of $\Acal_X \times \Aaff^1$ forming a flat and reduced family of expansions of $\Acal_X$. The central fibre $\Ecal_X$ has irreducible components $\Ecal_X(v)$ indexed by vertices of $\Gamma$ which meet along divisors $\Ecal_X(e)$ indexed by edges of $\Gamma$. \red{There is a natural logarithmic collapsing morphism given as the composite $\Ecal_X \to \Acal_X \times \Aaff^1 \to \Acal_X$.} We will construct a strongly transverse lift $g \colon C \to \Ecal_X$ of $f \colon C \to \Acal_X$.

\red{
For each vertex $v$ of $\Gamma$ the restriction of $f \colon C \to \Acal_X$ to the curve component $C_v$ factors through
\[ f_v \colon C_v \to \Acal_X(v) \]
where $\Acal_X(v) \subseteq \Acal_X$ is the closed stratum corresponding to $\sigma_v \preceq \Sigma$. The collapsing morphism restricts to a map of closed strata $\Ecal_X(v) \to \Acal_X(v)$. Restricting further to the locally closed strata, we obtain a principal torus bundle
\[ \Ecal^{\circ}_X(v) \to \Acal^{\circ}_X(v) \]
with structure group $T_{\sigma_v} = N_{\sigma_v} \otimes \Gm$. 

Consider the irreducible components of $\Dcal_X$ containing the stratum $\Acal_X(v)$. Let $L_1,\ldots,L_k$ denote the restrictions to $\Acal_X^{\circ}(v)$ of their associated line bundles. Then $\Ecal_X^{\circ}(v)$ is precisely the principal bundle associated to $\oplus_{i=1}^k L_i$ \cite[Proposition~4.3]{CarocciNabijou2}. Given a map to $\Acal_X^{\circ}(v)$, a lift to $\Ecal_X^{\circ}(v)$ is given by specifying nonvanishing sections of the pullbacks of $L_1,\ldots,L_k$.
}

\red{
Denote by $x_1,\ldots,x_r$ the special points in $C_v$. For each $L_i$ the combinatorial type of naive map determines contact orders $c_{ij} \in \Z$ for $1 \leq j \leq r$. By the balancing condition these data satisfy
\[ \Sigma_{j=1}^r c_{ij} = \deg f_v^\star L_i . \]
Since we work in genus zero, $C_v$ is a marked $\PP^1$ and so we may choose an isomorphism
\begin{equation} \label{eqn: isomorphism between Li and AJ line bundle} \OO_{C_v}(\Sigma_{j=1}^r c_{ij} x_j) \cong f_v^\star L_i.\end{equation}
The standard rational section on the left hand side produces a rational section of $f_v^\star L_i$. This section is both regular and nonvanishing away from the special points and hence we obtain a map
\[ C_v \setminus \{ x_1,\ldots,x_r \} \to \Ecal_X^{\circ}(v).\]
Passing to the partial compactification $\Ecal_X(v)$ of the codomain, the above map extends to all of $C_v$ and produces a strongly transverse lift}
\[
\begin{tikzcd}
C_v \arrow[dr, bend right,"f_v" left] \arrow[r,dashed,"g_v"] & \Ecal_X(v)\arrow{d}\\
& \Acal_X(v)
\end{tikzcd}
\]
with the desired tangency orders. The set of such lifts forms a torsor under the group $T_{\sigma_v}$ with moduli the choice of isomorphisms \eqref{eqn: isomorphism between Li and AJ line bundle}. See~\cite[Proposition~3.3.3]{RanganathanSkeletons1} for a parallel calculation.

We now argue that the scalar moduli above can be chosen in such a way that the lifts $g_v \colon C_v \to \Ecal_X(v)$ glue to form a globally defined map $g \colon C \to \Ecal_X$. The following argument uses slope-sensitivity in an essential way.

Choose a root vertex of $\Gamma$ and orient the graph to flow away from the root. We will construct $g$ by iterating along this flow. At each step, we traverse from $v_1$ to $v_2$ along an oriented edge $\vec{e}$ of the graph, and assume that we have already constructed a lift
\[
g_1: C_{v_1}\to \Ecal_X(v_1).
\]
Fix an arbitrary lift $g_2: C_{v_2}\to \Ecal_X(v_2)$ with the prescribed contact orders. We claim that there exists a scalar $\lambda$ in $T_{\sigma_{v_2}}$ such that
\begin{equation} \label{eqn: gluing at nodes} \lambda ( g_2(q_e) ) = g_1(q_e) \end{equation}
where $q_e \in C$ is the node connecting $C_{v_1}$ and $C_{v_2}$. Note that $g_2(q_e)$ is an interior point of the join divisor $\Ecal_X(e)$. This interior is a principal torus bundle over the interior of $\Acal_X(\sigma_e)$ with structure group $(N_{\sigma_e} / \ZZ m_{\vec{e}} ) \otimes \Gm$ where $m_{\vec{e}} \in N_{\sigma_e}$ is the slope along the oriented edge (see e.g. \cite[Theorem~1.8]{CarocciNabijou1}). To prove the claim it therefore suffices to show that the composition
\begin{equation} \label{eqn: composition of lattices} N_{\sigma_{v_2}} \to N_{\sigma_e} \to N_{\sigma_e} / \mathbb{Z} m_{\vec{e}} \end{equation}
is surjective. By Corollary~\ref{lem: small jumping} there are two cases to consider: $\dim \sigma_{v_2}=\dim \sigma_e$ and $\dim {\sigma_{v_2}}=\dim \sigma_e-1$. In the first case, \eqref{eqn: composition of lattices} is the composition of an isomorphism followed by a surjection and so the claim holds. In the second case, we have a hyperplane inclusion
\[
N_{\sigma_{v_2}}\hookrightarrow N_{\sigma_e}.
\]
Corollary~\ref{lem: slope negativity} implies that $m_{\vec{e}}$ does not belong to the image, and hence the composition \eqref{eqn: composition of lattices} is surjective as claimed. We conclude the existence of a scalar $\lambda$ satisfying \eqref{eqn: gluing at nodes}. Replacing $g_2$ by $\lambda \circ g_2$ we obtain a strongly transverse lift of $f_2$ which glues to $g_1$ along $q_e$. 

By iterating, we construct the desired global lift $g \colon C \to \Ecal_X$. Since this lift is strongly transverse, standard arguments (see e.g.~\cite[Section~6]{RangExpansions} or \cite[Section~4.2]{AbramovichChenGrossSiebertDegeneration}) produce a logarithmic enhancement $\Ccal \to \Ecal_X$. Composing with the collapsing map produces a logarithmic morphism $\Ccal \to (\Acal_X|\Dcal_X)$ and this completes the proof.
\end{proof}

\section{Logarithmic from absolute} 

\noindent The main result of this paper allows for the importation of orbifold techniques into logarithmic Gromov--Witten theory. As a first application, we establish the following conceptual result.

\begin{theorem}\label{thm: log from absolute} The genus zero logarithmic Gromov--Witten theory of $(X|D)$ is determined algorithmically from the genus zero absolute Gromov--Witten theories of all the strata of $(X|D)$.
\end{theorem}

\begin{proof} Let $(X|D)$ be a normal crossings pair and fix numerical data $\Lambda$ for a specific logarithmic Gromov--Witten invariant. Choose a $\Lambda$-sensitive modification $(X^\dag|D^\dag) \to (X|D)$. We may additionally assume that $D^\dag \subseteq X^\dag$ is \emph{simple} normal crossings. 

Theorem~\ref{thm: main-thm} equates the logarithmic Gromov--Witten theory of $(X|D)$ to the orbifold Gromov--Witten theory of $(X^\dag|D^\dag)$. The multi-root stack associated to $(X^\dag|D^\dag)$ may be factored as a sequence of root stacks with respect to smooth divisors. By iteratively applying \cite[Theorem~1.9]{ChenDuWang}, we conclude that the orbifold Gromov--Witten theory of $(X^\dag|D^\dag)$ is determined by the absolute Gromov--Witten theories of all the strata of $(X^\dag|D^\dag)$.

It remains to reduce these to the absolute Gromov--Witten theories of the strata of $(X|D)$. By induction it is sufficient to consider the case where $X^\dag \to X$ is a single blowup \red{along} a smooth stratum $Z \subseteq X$. Every stratum $V^\dag \subseteq X^\dag$ is then either a blowup of, or a projective bundle over, a stratum $V \subseteq X$. In the former case, work of Fan \cite[Theorem~B]{FanBlowups} reduces the theory of $V^\dag$ to the theories of $V$ and $V \cap Z$, both of which are strata in $(X|D)$. 

In the latter case, we observe that the projective bundle $V^\dag \to V$ is the projectivisation of a direct sum of line bundles. This is because the logarithmic structure provides a natural splitting of the normal bundle of the blowup centre $Z$. Therefore there is a well-defined big torus acting in the fibres of $V^\dag \to V$. Localising with respect to this action reduces the theory of $V^\dag$ to the theory of $V$ (analogously to the proof of Theorem~\ref{thm: orbifold theory from absolute theory} below). Similar arguments have been employed by others: see e.g. \cite{MaulikPandharipande,Brown}. \end{proof} 

If we restrict to insertions coming from $X$, then a slightly weaker form of the statement can be proved in a direct manner. While weaker, we believe that it transparently illustrates a general theme in the present paper, namely that once the problem has been moved to the orbifold side, more techniques become available to manipulate the theory. We include this weaker statement to give a sense for this. See also~\cite{TsengYouMirror} where similar ideas appear. 

As usual fix $(X|D=D_1+\ldots+D_k)$ a simple normal crossings pair, $\vec{r}=(r_1,\ldots,r_k)$ a tuple of pairwise coprime positive integers and $\Xcal = X_{D,\vec{r}}$ the associated multi-root stack. For $Y=\Xcal$ or $Y$ a stratum of $(X|D)$, there is a natural morphism $Y \to X$. The \textbf{restricted Gromov--Witten theory} of $Y$ is the system of descendant invariants with evaluation classes pulled back from the cohomology of $X$.

\begin{theorem} \label{thm: orbifold theory from absolute theory} The genus zero restricted Gromov--Witten theory of $\Xcal$ is determined algorithmically from the genus zero restricted absolute Gromov--Witten theories of all the strata of $(X|D)$.
\end{theorem}

\begin{proof} We recall the projective bundle construction of \cite[Section~3.1]{TsengYouMirror}. For $1 \leq i \leq k$ let $P_i = \PP_X(\OO_X(D_i) \oplus \OO_X)$. This contains zero and infinity sections $X_{i0},X_{i\infty}$ satisfying
\begin{align*} \OO_{P_i}(1)|_{X_{i0}} \cong \OO_{X}, \quad \OO_{P_i}(1)|_{X_{i\infty}} \cong \OO_X(D_i).\end{align*}
There is a tautological section $s_i \in H^0(P_i,\OO_{P_i}(1))$ and a natural isomorphism of pairs
\[ (V(s_i)|V(s_i) \cap X_{i\infty}) = (X|D_i).\]
From the $P_i$ we construct the following $(\PP^1)^k$-bundle over $X$
\[ P \colonequals P_1 \times_X P_2 \times_X \cdots \times_X P_k.\]
We let $p_i \colon P \to P_i$ be the projection and consider the multi-root stack
\[ \Pcal \colonequals \sqrt[\uproot{4} r_1]{(P,p_1^{-1}X_{1\infty})} \times_P \cdots \times_P \sqrt[\uproot{4} r_k]{(P,p_k^{-1}X_{k\infty})}.\]
This is a $\PP(1,r_1)\times\cdots\times\PP(1,r_k)$-bundle over $X$. Let $p \colon \Pcal \to P$ be the map to the coarse moduli space. Then the pullbacks of the tautological sections
\[ p^\star p_i^\star s_i \in H^0(\Pcal,p^\star p_i^\star \OO_{P_i}(1)) \]
satisfy
\[ V(p^\star p_1^\star s_1, \ldots, p^\star p_k^\star s_k) \cong \Xcal.\]
Since each bundle $p^\star p_i^\star \OO_{P_i}(1)$ is positive and pulled back from the coarse moduli space, we may apply the quantum Lefschetz theorem for orbifolds \cite{CoatesEtcOrbifoldLefschetz}. This reduces the restricted genus zero theory of $\Xcal$ to the genus zero theory of $\Pcal$ capped with the class
\begin{equation} \label{eqn: Euler class term} \prod_{i=1}^k \e(\Rder \pi_\star f^\star p^\star p_i^\star \OO_{P_i}(1)).\end{equation}
We compute this by localising with respect to the big torus acting on the fibres of $\Pcal \to X$. Each fixed locus $\Zcal_I \subseteq \Pcal$ in the target is indexed by a subset $I \subseteq \{1,\ldots,k\}$ and is isomorphic to a gerbe over $X$ banded by $\Pi_{i \in I} \mu_{r_i}$
\begin{equation*} \label{eqn: gerbe over X} \Zcal_I =  \prod_{i \in I} \sqrt[\uproot{4} r_i]{(X,\OO_X(D_i))} \times_{X^I} X.\end{equation*}
Now consider a fixed locus in the moduli space $\Orb_\Lambda(\Pcal)$. Up to a finite cover, this is a fibre product of spaces of maps to the gerbes $\Zcal_I$. The fibring is with respect to the evaluation maps to the coarse moduli space $X$.

Since $\OO_{P_i}(1)|_{X_{i0}} \cong \OO_X$ it follows that for $i \not\in I$ the class $\e(\Rder \pi_\star f^\star p^\star p_i^\star \OO_{P_i}(1))$ is pure weight. On the other hand for $i \in I$ we have
\[ p^\star p_i^\star \OO_{Y_i}(1) \cong s^\star \OO_X(D_i) \]
where $s \colon \Zcal_I \to X$ is the rigidification. Let $u \colon \Orb_\Lambda(\Zcal_I) \to \Mup_\Lambda(X)$ denote the induced morphism on spaces of stable maps. By the standard comparison of universal curves (see e.g. Section~\ref{sec: orbifold is naive}) we obtain
\[ \e(\Rder \pi_\star f^\star p^\star p_i^\star \OO_{P_i}(1)) = u^\star \e(\Rder \pi_\star f^\star \OO_X(D_i)). \]
Finally, the normal bundle term expands as a power series in equivariant weights and (coarse) psi classes \cite{MLiuLocalisation}. For a discussion of the comparison between coarse and gerby psi classes see \cite[Section~8.3]{AGVDMStacks}.

In summary, the contribution of the fixed locus is a (product of) integrals over $\Orb_\Lambda(\Zcal_I)$ with insertions pulled back along $u$. By the virtual push-forward theorem for root gerbes \cite[Theorem~4.3]{AJTGerbes1} this can be expressed as an integral over $\Mup_\Lambda(X)$. The integrand is a product of equivariant weights, evaluation classes, psi classes, and an Euler class term. But the Euler class term is simply
\[ \prod_{i \in I} \e(\Rder \pi_\star f^\star \OO_X(D_i)) = \iota_\star [\Mup_\Lambda( \cap_{i \in I} D_i)]^{\virt} \]
by the absolute quantum Lefschetz theorem. We have therefore successfully reduced the restricted theory of $\Xcal$ to the restricted theories of the strata $\cap_{i \in I}D_i$ of $(X|D)$.\end{proof}

\footnotesize
\bibliographystyle{alpha}
\bibliography{Bibliography.bib}
\,

\noindent Luca Battistella. Goethe Universit\"at Frankfurt. \href{mailto:battistella@math.uni-frankfurt.de}{battistella@math.uni-frankfurt.de}\smallskip

\noindent Navid Nabijou. Queen Mary University of London. \href{mailto:n.nabijou@qmul.ac.uk}{n.nabijou@qmul.ac.uk}\smallskip

\noindent Dhruv Ranganathan. University of Cambridge. \href{mailto:dr508@cam.ac.uk}{dr508@cam.ac.uk}

\end{document}